\documentclass[reqno,11pt]{amsart}
\usepackage{cases}
\usepackage{mathrsfs}
\usepackage[T1]{fontenc}
\usepackage{mathrsfs}
\usepackage{amsmath,latexsym,amssymb,amsfonts,amsbsy, amsthm}
\usepackage[usenames]{color}
\usepackage{xspace,colortbl}
\usepackage{epsfig}
\usepackage{amsmath,amsfonts,amsthm,amssymb,amscd}

\usepackage{esint}
\usepackage{graphicx}
\usepackage{caption2}
\usepackage{subfigure}
\usepackage{float}

\usepackage{lscape}

\usepackage{color}

\input amssym.def
\input amssym.tex
\newtheorem{theorem}{Theorem}[section]
\newtheorem{remark}[theorem]{Remark}
\newtheorem{lemma}[theorem]{Lemma}
\newtheorem{definition}[theorem]{Definition}
\newtheorem{cor}[theorem]{Corollary}

\newcommand{\ben}{\begin{eqnarray}}
\newcommand{\een}{\end{eqnarray}}
\newcommand{\beno}{\begin{eqnarray*}}
\newcommand{\eeno}{\end{eqnarray*}}

\def\i{\,{\rm i}\,}

\newdimen\eqjot \eqjot = 1\jot
\def\openupeq{\openup \the\eqjot}
\def\addtab#1={#1\;&=}
\def\addtabe#1=#2={#1=#2\;&=}

\def\ostrut#1#2{\hbox{\vrule height #1pt depth #2pt width 0pt}}

\begin{document}

\title [Bidirectional dispersive equations]
{New revival phenomena for  bidirectional dispersive hyperbolic equations}

\author{George Farmakis}
\address{ George Farmakis \newline Department of Mathematics and Maxwell Institute for Mathematical Sciences, Heriot-Watt
University, Edinburgh EH14 4AS; London South Bank University, London SE1 6LN, UK}
\email{george.farmakis@lsbu.ac.uk}
\author{Jing Kang}
\address{Jing Kang\newline
Center for Nonlinear Studies and School of Mathematics, Shaanxi Key Laboratory of Mathematical Theory and Computations of Fluid Mechanics, Northwest University, Xi'an 710069, P.R. China}
\email{jingkang@nwu.edu.cn}
\author{Peter J. Olver}
\address{Peter J. Olver\newline
School of Mathematics, University of Minnesota, Minneapolis, MN 55455, USA}
\email{olver@umn.edu}
\author{Changzheng Qu}
\address{Changzheng Qu\newline
School of Mathematics and Statistics, Ningbo University, Ningbo 315211, P.R. China; Shaanxi Key Laboratory of Mathematical Theory and Computations of Fluid Mechanics, Northwest University, Xi'an, 710069, P. R. China}
\email{quchangzheng@nbu.edu.cn}
\author{Zihan Yin}
\address{Zihan Yin\newline
School of Mathematics, Northwest University, Xi'an 710069, P.R. China}
\email{yinzihan@stumail.nwu.edu.cn}

\begin{abstract}
In this paper, the dispersive revival and fractalization phenomena for bidirectional dispersive equations on a bounded interval subject to periodic boundary conditions and discontinuous initial profiles are investigated. Firstly,  we study the periodic initial-boundary value problem of the linear beam equation with step function initial data, and analyze the manifestation of the revival phenomenon for the corresponding solution at rational times. Next, we extend the investigation to periodic initial-boundary value problems of more general bidirectional dispersive equations. We prove that, if the initial functions are of bounded variation, the dynamical evolution of such periodic problems depend essentially upon the large wave number asymptotics of the associated dispersion relations. Integral polynomial or asymptotically integral polynomial dispersion relations produce dispersive revival/fractalization rational/irrational dichotomies, whereas those with non-polynomial growth result in fractal profiles at all times. Finally, numerical experiments, in the concrete case of the nonlinear beam equation, are used to demonstrate how such effects persist into the nonlinear regime.

\end{abstract}

\maketitle \numberwithin{equation}{section}



\section{Introduction}

This paper is devoted to the study of the periodic initial-boundary value problem for bidirectional dispersive partial differential equations. We prove that, for linear equations, if the initial condition at time zero is a step function or, more generally, a function of bounded variation, the time evolution of the  bidirectional dispersive  equations subject to periodic boundary conditions will exhibit new revival phenomena at rational times,  of a different form from that previously observed in unidirectional dispersive evolution equations, whereas at irrational times the solution exhibits a continuous, but non-differentiable fractal profile.

The term ``revival'' is based on the experimentally observed phenomenon of quantum revival \cite{BMS, VVS}, in which an electron that is initially concentrated near a single location of its orbital shell is re-concentrated near a finite number of orbital locations at certain times. A precursor of the revival phenomenon was observed as far back as 1834 in a striking optical experiment, \cite{Tal}, conducted in 1836 by William Henry Fox Talbot. This motivated the pioneering work of Berry and his collaborators, \cite{Ber, BK, BMS}, on what they called the Talbot effect in the context of the linear free space Schr\"{o}dinger equation. Rigorous analytical results and estimates justifying the Talbot effect can be found  in the work of Kapitanski and Rodnianski, \cite{KR, Rod}, Oskolkov, \cite{Osk98, Osk92}, and Taylor, \cite{Tay}.
The Talbot effect governs, in the quantum mechanical setting, the behavior of rough solutions subject to periodic boundary conditions. The evolution of the rough initial profile, for instance, a step function, also known as the Riemann problem \cite{Whi}, ``quantizes'' into a dispersive revival profile at rational times, but ``fractalizes'' into a continuous but nowhere differentiable profile having a specific fractal dimension at irrational times.

In \cite{CO12, Olv10}, the same Talbot effect, which the authors called dispersive quantization and fractalization, was shown to appear in general periodic linear dispersive equations possessing an ``integral polynomial'' (a polynomial with integer coefficients) dispersion relation, which included the prototypical linearized Schr\"{o}dinger and Korteweg-de Vries (KdV) equations. Based on these investigations, one learns that a linear dispersive equation admitting a polynomial dispersion relation and subject to periodic boundary conditions will exhibit the revival phenomenon at each rational time, which means that the fundamental solution, i.e., that induced by a delta function initial condition, localizes into a finite linear combination of delta functions. This has the remarkable consequence that the solution, to any initial value problem, at rational times is a finite linear combination of translates of the initial data and hence its value at any point on the periodic domain depends only  upon finitely many of the initial values.  In \cite{OSS},  the revival phenomenon for the linear free space Schr\"{o}dinger equation subject to pseudo-periodic boundary conditions was investigated, see also \cite{BFP} for the same model and for the quasi-periodic linear KdV equation. In \cite{BOPS}, a more general revival phenomenon, that produces dispersively quantized cusped solutions of the periodic Riemann problem for three linear integro-differential equations, including the Benjamin-Ono equation, the Intermediate Long Wave equation and the Smith equation were studied. More recently, these phenomena were shown to extend to multi-component dispersive equations, see \cite{YKQ}.  For a class of two-component linear systems of dispersive evolution equations, the dispersive quantization conditions, which may yield quantized structures for step-function initial value at rational times, are provided.

Inspired by these linear results, the phenomena of dispersive quantization and fractalization for the periodic Riemann problem for nonlinear dispersive evolution equations on periodic domains, including the integrable nonlinear Schr\"{o}dinger (NLS), KdV and modified KdV (mKdV) equations as well as non-integrable versions with higher-order nonlinearities were studied numerically in  \cite{CO14}. Erdo$\mathrm{\breve{g}}$an, Tzirakis and their collaborators established rigorous results on the fractalization for the nonlinear equations at a dense set of times. Quantifying the irrational time fractalization in terms of the estimate on the fractal dimension, their results, on the one hand extend the results of Oskolkov and Rodnianski to a class of nonlinear integer polynomial dispersive equations subject to initial data of bounded variation, and, on the other hand, confirm the numerical observations of fractalization in \cite{CO14}.
Erdo\'{g}an and Tzirakis studied the cubic NLS and KdV equations on a periodic domain with initial data of bounded variation in \cite{ET13-nls} and \cite{ET13-kdv}, respectively. Subsequently, together with Chousionis, they obtained  some results on the Minkowski dimension of the fractalization profiles for dispersive linear partial differential equations with monomial dispersion relation \cite{CET}. We refer the reader to the survey texts \cite{ES19, ET16} for irrational time fractalization results. See also the recent survey \cite{Smi}.

To date, investigations have almost all concentrated on unidirectional dispersive systems. In the present paper, we will show that the dispersive revival/fractalization rational/irrational dichotomy extends to bidirectional dispersive equations of the form
\begin{equation}\label{bi-eq-in}
u_{tt}=L[u],
\end{equation}
where $L$ is a scalar differential operator with constant coefficients. Obviously, equation \eqref{bi-eq-in} is equivalent to the following two-component evolutionary system
\begin{equation}\label{2comp-ev}
u_{t}=v,\qquad v_t=L[u],
\end{equation}
which, however, does not satisfy the dispersive quantization conditions given in \cite{YKQ}.
As we describe below, in the bidirectional setting, if we set the initial conditions equal to the same step function, then the solution of the corresponding periodic Riemann problem will exhibit  qualitative dispersive quantization behaviour, of a different form than the standard piecewise constant   solutions admitted by the unidirectional systems, such as the linear KdV and Schr\"{o}dinger equations and their associated multi-component generalizations.
Interestingly, in the concrete case of the linear beam equation
\begin{equation}\label{beam-eq}
u_{tt}+u_{xxxx}=0,
\end{equation}
these solutions  at rational times $t^\ast=\pi p/q$ with $q> 2$ appear to be piecewise parabolic, non-constant between jump discontinuities, whereas at the rational times $t_k^0=\pi(2k-1)/2$, the solution becomes a continuously differentiable curve, with analytical expression \eqref{sol-pi2}.

Similar studies were initiated in one of the authors' recent Ph.D. thesis \cite{Fa}, which studies the revival property in bidirectional dispersive equations \eqref{bi-eq-in} where the operator $L$ is an even-order poly-Laplacian, which includes the linear wave equation and the beam equation \eqref{beam-eq}, subject to periodic and quasi-periodic boundary conditions. We should further mention that from a general perspective, the form of the revival effect in the periodic bidirectional problems considered here resembles this of the revival effect in the free linear Schr\"{o}dinger equation with Robin boundary conditions $bu(t,0) = (1-b)u_{x}(t,\pi)$, where $b\in (0,1)$ is a parameter, see \cite{BFP}. Indeed, in both cases the solution at rational times is given as the sum of the revival of the initial condition and a more regular function, which can be considered as a weak type of revival. Other models in the literature that exhibit such weak revivals include the periodic cubic NLS and KdV equations \cite{ET13-nls,ET13-kdv}, the periodic linear Schr\"{o}dinger equation with periodic potential \cite{Rod-potential,CKKKS} and the linear Schr\"{o}dinger equation subject to Dirichlet boundary conditions \cite{BFP-Pot}.

The linear beam equation, which is studied in Section 2, is a typical example of a model with the dispersion relation of the form $\omega(k)=\pm k^N,\;2\leq N\leq \mathbb{Z}^+$, namely when $N=2$. Although it is a special case, it motivates the study of the more general case and illustrates the idea and the method in the proof. More importantly, through the analysis and derivation of its periodic initial-boundary value problem, we arrive at some classical results for the Riemann zeta function. It implies that  the periodic initial-boundary value problems for such systems  can provide an alternative mechanism for establishing such classical identities.

The three concrete goals of the present paper are as follows. The first is to investigate the new phenomenon of dispersive revival in greater detail, by examining the periodic initial-boundary value problems for bidirectional dispersive equations. We will provide an explicit characterization of the solution profiles of the periodic Riemann problem for the linear beam equation, leading to the general form of dispersive revival for bidirectional periodic initial-boundary value  problems with various dispersion relations, including  integral polynomial and non-polynomial. Our main results contain the analytic description of the new phenomena of the dispersive revival, which can be found in Section 2 for the linear beam equation, and Section 3 for general bidirectional equations, respectively. In the particular case of monomial dispersion relations, we present an alternative approach. Secondly, with the aim to show that such effects can persist into the nonlinear regime, we present numerical simulations, based on the Fourier spectral method, of the periodic Riemann problem for the nonlinear beam equation in Section 4. Numerical approximation supplies strong evidence that the dispersive revival/fractalization rational/irrational dichotomy persists into the nonlinear regime, whereas, when compared with the unidirectional systems, the nonlinear terms induce greater variations of the curve profiles, including their convexities.
Finally, in the course of our analysis, we find that the solutions at rational times of the periodic Riemann problem for bidirectional dispersive  equations with integral polynomial dispersion relations are closely related to identities for the Riemann zeta function, which is of great importance in analytic number theory. In summary, these new revival phenomena warrant further investigation, both mathematically, and to develop its profound applications.

\newpage

\section{Revival for the linear beam equation}

The starting point is the periodic initial-boundary value problem for the linear beam equation on the interval $0\leq x \leq 2\pi$:
\begin{equation}\label{ibv-beam}
\begin{cases}
u_{tt}+u_{xxxx}=0,\\
u(0, x)=f(x), \qquad \qquad  u_t(0, x)=g(x),\\
{\displaystyle\partial_{x}^{j}u(t,0) = \partial_{x}^{j}u(t,2\pi)\atop \displaystyle\partial_{x}^{j}\partial _tu(t,0) = \partial_{x}^{j}\partial _tu(t,2\pi),} \qquad  j =0, \ 1, \ 2, \ 3.
\end{cases}
\end{equation}
The linear beam equation is a bidirectional dispersive  equation modeling small vibrations of a thin elastic beam, with quadratic dispersion relation $\omega(k)=\pm k^2$. We focus our attention on the initial data given by a step function:
  \begin{equation}\label{iv-s}
 f(x)=g(x)=\sigma(x)=
\begin{cases}
-1 ,\qquad & 0\leq x<\pi,\\
1 ,\qquad & \pi\leq x<2\pi,
\end{cases}
\end{equation}known as the \emph{Riemann problem}.
Without further mention, here and elsewhere below, we assume that functions and distributions defined on $[\,0, \,2\pi \,]$ are extended $2\pi$-periodically to $\mathbb{R}$ in the usual way, when required.

For the solution of \eqref{ibv-beam} with initial data \eqref{iv-s}, we have the following result.

\begin{lemma}\label{beam-solut}
The periodic initial-boundary value problem \eqref{ibv-beam}-\eqref{iv-s} has the following solution
\begin{equation}\label{sol1-beam}
u(t, x)=-\frac{4}{\pi}\left(\sum_{n=0}^{+\infty}\frac{\cos((2n+1)^2 t)\sin((2n+1) x)}{2n+1}+\sum_{n=0}^{+\infty}\frac{\sin((2n+1)^2 t)\sin((2 n+1)x)}{(2n+1)^3}\right).
\end{equation}
\end{lemma}

Throughout the paper, a time $t>0$ will be designated as \emph{rational} if $t/\pi \in \mathbb{Q}$, i.e., $t=t^\ast=\pi p/q$, with $p$ and $0\neq q\in \mathbb{Z}^+$ having no common factors. Otherwise, if $t/\pi \notin \mathbb{Q}$, the time is called \emph{irrational}.
To analyze the qualitative behavior of solution \eqref{sol1-beam} at the rational times, we invoke the following Lemma.

\begin{lemma}\label{beam-le1}
Given $j,q\in \mathbb{Z^+}$ with $q \ne 0$, let $\sigma^{j, q}(x)$ be the box function defined as
\begin{equation}\label{box}
\sigma^{j, q}(x)=
\begin{cases}
\,1,\qquad &\displaystyle\frac{\pi j}{q} \leq x < \frac{\pi(j+1)}{q},\quad 0\leq j\leq 2q-1,\\
\,0,\qquad &otherwise.
\end{cases}
\end{equation}
Let $N,\,p,\,q\in \mathbb{Z^+},\,N\geq 2,\,q\neq 0$. Then, the following formulae hold.

{\rm (}{\it i\/}{\rm )}
\begin{equation}\label{eq1-l1}-\frac{4}{\pi}\sum_{n=0}^{+\infty}\frac{\displaystyle\cos\left((2n+1)^N \frac{\pi p}{q}\right)}{2n+1}\sin((2n+1) x)=\sum_{j=0}^{2q-1}a_j\left(\frac{p}{q}\right)\sigma^{j, q}(x),\end{equation}

{\rm (}{\it ii\/}{\rm )}
For each even $N$,
\begin{equation}\label{eq2-l1}-\frac{4}{\pi}\sum_{n=0}^{+\infty}\frac{\displaystyle\sin\left((2n+1)^N \frac{\pi p}{q}\right)}{2n+1}\sin((2n+1) x)=\sum_{j=0}^{2q-1}b_j\left(\frac{p}{q}\right)\sigma^{j, q}(x),\end{equation}

{\rm (}{\it iii\/}{\rm )}
For each odd  $N$,
\begin{equation}\label{eq3-l1}-\frac{4}{\pi}\sum_{n=0}^{+\infty}\frac{\displaystyle\sin\left(((2n+1)^N \frac{\pi p}{q}\right)}{2n+1}\cos((2n+1) x)=\sum_{j=0}^{2q-1}\tilde{b}_j\left(\frac{p}{q}\right)\sigma^{j, q}(x),\end{equation}
where $a_j, b_j, \tilde{b}_j \in \mathbb{R}$, $j=0,\ldots, 2q-1$, are certain constants which depend on $p$ and $q$.

\end{lemma}

To prove Lemma \ref{beam-le1},  we need the following theorem, which is based on Theorem 3.2 and Corollary 3.4 in \cite{CO12}, and underlies the dispersive quantization effect for equations with ``integral polynomial'' dispersion relation.

\begin{definition}\label{in-po}
A polynomial $P(k)=c_0+c_1k+\cdots+c_Nk^N$ is called an integral polynomial if its coefficients are integers: $c_i\in \mathbb{Z}, i=0, \ldots, N$.
\end{definition}

\begin{theorem}\label{thm-olv1}
Suppose that the dispersion relation of the evolution equation $u_t=L[u]$ is an integral polynomial\/{\rm:}
\begin{equation*}
\omega(k)= P(k).
\end{equation*}
 Then at every rational time $t^\ast=\pi p/q$, with $p$ and $0\neq q\in \mathbb{Z}$, the fundamental solution
\begin{equation*}
F(t, x)\sim \frac{1}{2\pi}\sum_{k=-\infty}^{\infty}e^{\i(k x+\omega(k)t)}
\end{equation*}
is  a linear combination of $q$ periodically extended delta functions concentrated at the rational nodes $x_j=\pi j/q$ for $j\in \mathbb{Z}$. Moreover, at $t^\ast=\pi p/q$, the solution profile to the periodic initial-boundary value problem for $u_t=L[u]$ is  a linear combination of $\leq 2q$ translates of its initial data $u(0, x)=f(x)$, i.e.,
\begin{equation*}
u\left(\frac{\pi p}{q}, x\right)=\sum_{j=0}^{2q-1}a_j\left(\frac{p}{q}\right)f\left(x-\frac{\pi j}{q}\right).
\end{equation*}

\end{theorem}

\begin{remark}
 Slightly more generally, the dispersion relation can be a nonzero multiple of an integral polynomial. By suitably rescaling time, the stated result holds, the only difference being which times are designated as rational or irrational.  A similar remark holds if one rescales space to consider the equation on a different spatial interval.
\end{remark}

\begin{proof}[\bf{Proof of Lemma \ref{beam-le1}}]
According to Theorem \ref{thm-olv1}, if the underlying equation admits a dispersion relation
$\omega(k)=\pm k^N$, and the initial data is the unit step function
\begin{equation*}
 u(0, x)=
\begin{cases}
0 ,\qquad & 0\leq x<\pi,\\
1 ,\qquad & \pi\leq x<2\pi,
\end{cases}
\end{equation*}
with Fourier coefficients
\begin{eqnarray*}
c_k=
\begin{cases}
\,\displaystyle\frac{1}{2},\qquad &k=0,\\
\,0,\qquad &k\neq0\; \mathrm{even},\\
\, \displaystyle\frac{\i}{\pi k},\qquad &k\; \mathrm{odd}.
\end{cases}
\end{eqnarray*}
Then, at a rational time  $t^\ast=\pi p/q$,  the corresponding solution has the Fourier series form
\begin{equation*}
u^\pm(t, x)=\sum_{k=-\infty}^{+\infty}c_ke^{\i (kx\pm k^N t)}
\end{equation*}
and hence is constant on every subinterval $\pi j/q< x<\pi(j+1)/q$, for $j=0, \ldots, 2q-1$, namely
\begin{equation}\label{sol-l1}
u^\pm(t^\ast, x)=\sum_{j=0}^{2q-1}\gamma_j^\pm\left(\frac{p}{q}\right)\sigma^{j, q}(x),
\end{equation}
for certain $\gamma_j^\pm\in \mathbb{C}$, $j=0, \ldots, 2q-1$, dependent on $p$ and $q$.
We thus need to distinguish two cases:

 {\bf Case 1.}\quad $N$ is even. It is easy to see that
\begin{eqnarray*}\begin{aligned}
u^+(t^\ast, x)+u^-(t^\ast, x)&=\sum_{k=-\infty}^{+\infty}c_k\left(e^{\i (kx+ k^N t^\ast)}+e^{\i (kx- k^N t^\ast)}\right)\\
&=1-\frac{4}{\pi}\sum_{n=0}^{+\infty}\frac{\cos\left((2n+1)^N t^\ast\right)}{2n+1}\sin((2n+1) x),\\
u^+(t^\ast, x)-u^-(t^\ast, x)&=\sum_{k=-\infty}^{+\infty}c_k\left(e^{\i (kx+ k^N t^\ast)}-e^{\i (kx- k^N t^\ast)}\right)\\
&=-\frac{4\i}{\pi}\sum_{n=0}^{+\infty}\frac{\sin\left((2n+1)^N t^\ast\right)}{2n+1}\sin((2n+1) x).
\end{aligned}\end{eqnarray*}

 {\bf Case 2.}\quad $N$ is odd. Similar to the above, we have
\begin{eqnarray*}\begin{aligned}
u^+(t^\ast, x)+u^-(t^\ast, x)&=\sum_{k=-\infty}^{+\infty}c_k\left(e^{\i (kx+ k^N t^\ast)}+e^{\i (kx- k^N t^\ast)}\right)\\
&=1-\frac{4}{\pi}\sum_{n=0}^{+\infty}\frac{\cos\left((2n+1)^N t^\ast\right)}{2n+1}\sin((2n+1) x),\\
u^+(t^\ast, x)-u^-(t^\ast, x)&=\sum_{k=-\infty}^{+\infty}c_k\left(e^{\i (kx+ k^N t^\ast)}-e^{\i (kx- k^N t^\ast)}\right)\\
&=-\frac{4}{\pi}\sum_{n=0}^{+\infty}\frac{\sin\left((2n+1)^N t^\ast\right)}{2n+1}\cos((2n+1) x).
\end{aligned}\end{eqnarray*}

Finally, these formulae, together with \eqref{sol-l1} yield \eqref{eq1-l1}-\eqref{eq3-l1}, respectively, proving the lemma.

\end{proof}

Denote the solution \eqref{sol1-beam} as
\begin{eqnarray}\label{beam-solut-1}
u(t,x):=\mathrm{I}(t, x)+\mathrm{II}(t, x).
\end{eqnarray}
Note that  \eqref{eq1-l1} implies that the first summation  in \eqref{beam-solut-1} evaluated at rational times $t^\ast$ has the representation
\begin{equation}\label{eq-s1-l1}
\mathrm{I}(t^\ast, x)=\sum_{j=0}^{2q-1}a_j\left(\frac{p}{q}\right)\sigma^{j, q}(x),
\end{equation}
for certain constants $a_0, \ldots, a_{2q-1}$ determined by  \eqref{eq1-l1}.
In particular, if $q=2$, at the corresponding specific rational time $t^0_k=\pi(2k-1)/2,\,k\in \mathbb{Z^+}$, it vanishes identically.

When it comes to  $\mathrm{II}(t^\ast, x)$, firstly, a direct computation shows that for any $t>0$ we have that
\begin{eqnarray}\begin{aligned}\label{eq-s2-l1}
&\frac{\sin((2n+1)^2 t)\sin((2 n+1)x)}{(2n+1)^3}\\
&\hskip.7in =-\int_{0}^{x}\int_{0}^{y}\frac{\sin((2n+1)^2t)\sin((2n+1)z)}{2n+1}\;\mathrm{d}z\mathrm{d}y+\frac{\sin((2n+1)^2t)}{(2n+1)^2} x.
\end{aligned}\end{eqnarray}
Next, again thanks to Lemma 2.1, we find that the first  of the two components in $\mathrm{II}(t, x)$ satisfies,
\begin{eqnarray}\begin{aligned}\label{II(1)-t1}
\mathrm{II}^{(1)}(t^\ast, x)&=\frac{4}{\pi}\int_{0}^{x}\int_{0}^{y}\sum_{n=0}^{+\infty}\frac{\sin((2n+1)^2 t^\ast)\sin((2n+1)z)}{2n+1}\;\mathrm{d}z\mathrm{d}y\\
&=-\int_{0}^{x}\int_{0}^{y}\sum_{j=0}^{2q-1}b_j\left(\frac{p}{q}\right)\sigma^{j, q}(z)\;\mathrm{d}z\mathrm{d}y,
\end{aligned}\end{eqnarray}
for certain constants $b_j,\,j=0,\ldots, 2q-1$ determined by \eqref{eq2-l1}.
We denote
\begin{equation*}
F(y)=\int_{0}^{y}\sum_{j=0}^{2q-1}b_j\sigma^{j, q}(z)\;\mathrm{d}z,\quad\mathrm{for}\quad 0\leq y\leq x, \end{equation*}
and set
\begin{equation*}
H(x)=\int_{0}^{x}F(y)\mathrm{d}y. \end{equation*}
It is easy to verify that
\begin{equation*}
F(y)=
\begin{cases}
b_0 y, \quad \quad &0\leq y\leq \frac{\pi }{q},\\
\displaystyle b_j y+\frac{\pi}{q}\sum\limits_{m=0}^{j-1}b_m-\frac{\pi}{q}j b_j, \quad  & \displaystyle\frac{\pi}{q}j\leq y \leq \frac{\pi}{q}(j+1),\quad j=1,\ldots, 2q-1,
\end{cases}
\end{equation*}
and hence
\begin{equation}\label{Hx}
H(x)
=
\begin{cases}
\frac{1}{2}b_0 x^2, \quad \quad &\displaystyle 0\leq x\leq \frac{\pi }{q},\ostrut0{13}\\
\frac{1}{2}b_1 x^2+\displaystyle\frac{\pi}{q}(b_0-b_1)x+\frac{\pi^2}{2q^2}(b_1-b_0),\quad &\displaystyle\frac{\pi }{q}\leq x\leq \frac{2\pi }{q}, \ostrut0{15}\\
\frac{1}{2}b_j x^2+h_1x+h_0,\quad &\displaystyle\frac{\pi }{q}j\leq x\leq \frac{\pi }{q}(j+1), \quad j=1,\ldots, 2q-1,
\end{cases}
\end{equation}
where
\begin{equation*}
h_1=\frac{\pi}{q}\left(\sum_{m=0}^{j-1}b_m-mb_m\right)\quad \mathrm{and}\quad h_0=\frac{\pi^2}{q^2}\left[\>\sum_{m=1}^{j-1}\left(\sum\limits_{i=0}^{m-1}b_i+\frac{b_m}{2}\right)+\frac{j^2}{2}b_j-j\sum_{m=0}^{j-1}b_m+\frac{b_0}{2}\>\right].\end{equation*}

This, when combined with \eqref{eq-s1-l1} and \eqref{eq-s2-l1}, allows us to arrive at the exact result for the
 solution \eqref{beam-solut} at the rational times,  which is summarized in the following theorem.

\begin{theorem}\label{thm-beam}
At a rational time $t^\ast=\pi p/q$, the solution to the  periodic initial-boundary value problem \eqref{ibv-beam} for the linear beam equation with the step function  initial datum \eqref{iv-s} takes the form
\begin{equation}\label{sol-ra-beam}
u(t^\ast, x)=\sum_{j=0}^{2q-1}a_j\left(\frac{p}{q}\right)\sigma^{j, q}(x)-H(x)+C(t^\ast)x,
\end{equation}
where
\begin{equation}\label{C}
C(t^\ast)=-\frac{4}{\pi}\sum_{n=0}^{+\infty}\frac{\sin((2n+1)^2t^\ast)}{(2n+1)^2},
\end{equation}
$\sigma^{j, q}(x)$ is the box function defined in \eqref{box}, $H(x)$ is the piecewise quadratic function defined in \eqref{Hx}, and $a_j$, $j=0,\ldots, 2q-1$, are certain constants determined by equation \eqref{eq1-l1}.
\end{theorem}

With the explicit expression \eqref{sol1-beam} of the solution in hand, we now analyze its qualitative behaviour. First of all,  as a direct corollary of the estimate on the solution of the linearized dispersion equation given by Oskolkov \cite{Osk92} and Rodnianski \cite{Rod} ---  see also \cite{ET16} --- one has, for almost all irrational $t/(2\pi)$, the first summation
$\mathrm{I}\in\bigcap_{\epsilon>0}C^{\frac{1}{2}-\epsilon}$, which  in turn indicates that the second one $\mathrm{II}\in\bigcap_{\epsilon>0}C^{\frac{5}{2}-\epsilon}$. We thus conclude that, at the irrational times, the profile of the solution \eqref{beam-solut} is a continuous fractal, with fractal dimension $D=3/2$. When it comes to the rational times, note that in  the expression of the solution \eqref{sol-ra-beam}, on the one hand, $H(x)\in C^1$, on the other hand, $\mathrm{I}(t^\ast, x)$ is a piecewise constant with $\leq 2q$ discontinuities, apart from the specific times $t^0_k=\pi(2k-1)/2, \,k\in\mathbb{Z^+}$, which will result in $\mathrm{I}(t^0_k, x)\equiv0$. Therefore, we need to distinguish two cases:

{\bf Case 1.} \; $q=2$. At times $t^0_k=\pi(2k-1)/2, \,k\in\mathbb{Z^+}$,
\begin{equation*}
u(t^0_k, x)=-H(x)+(-1)^k\frac{4}{\pi}\sum_{n=0}^{+\infty}\frac{1}{(2n+1)^2}x\in C^1.
\end{equation*}
Referring back to the relation \eqref{eq2-l1} corresponding to $q=2, p=2k-1$, and $N=2$,  and solving for $b_j$ by making use of the  inverse discrete Fourier transform (IDFT) gives rise to
\begin{equation}\label{b-2-pi2}
b_0=b_1=(-1)^{k},\quad b_2=b_3=(-1)^{k-1}.
\end{equation}
Then, \eqref{Hx} readily leads to
\begin{eqnarray}
\begin{aligned}\label{Hx-ex1}
H_k(x)
&=(-1)^{k-1}
\begin{cases}
-\frac{1}{2} x^2, \quad \quad &0\leq x \leq \pi,\\
\frac{1}{2}x^2-2\pi x+\pi^2,\quad &\pi \leq x \leq 2\pi,
\end{cases}
\end{aligned}\end{eqnarray}
which, together with the fact that the solution \eqref{sol1-beam} is $2\pi$-periodic, i.e.,  $u(t^0_k, 0)=u(t^0_k, 2\pi)$, yields
\begin{equation}\label{C-pi2}
\sum_{n=0}^{+\infty}\frac{1}{(2n+1)^2}=\frac{\pi^2}{8}.
\end{equation}
It follows that, evaluated at each  $t^0_k=(2k-1)\pi/2, \,k\in\mathbb{Z^+}$,  the solution  can be written as
\begin{eqnarray}\begin{aligned}\label{sol-pi2}
u(t^0_k, x)
=(-1)^{k-1}
\begin{cases}
\frac{1}{2} (x^2-\pi x), \quad \quad &0\leq x \leq\pi,\\
-\frac{1}{2}(x^2-3\pi x+2\pi^2),\quad &\pi \leq x \leq 2\pi.
\end{cases}
\end{aligned}\end{eqnarray}
More interestingly, we find that the conclusion \eqref{C-pi2} agrees with the classical result
\begin{equation}\label{zeta-2}
\zeta (2) = \sum_{n=0}^{+\infty}\frac{1}{n^2}=\frac{\pi^2}{6},
\end{equation}
where
\begin{equation}\label{zeta}
\zeta(s)=\sum_{n=0}^{+\infty}\frac{1}{n^s},\quad s>1
\end{equation}
is the Riemann zeta function. The above procedure provides an alternative mechanism for establishing such classical identities, and the behavior of these solutions at rational times has intriguing connections with such number-theoretic exponential sums.

{\bf Case 2.}  $q\neq 2$. According to \eqref{sol-ra-beam}, at each rational time $t^\ast=\pi p/q, \,q\neq 2$,  the solution consists of a piecewise constant function, which is constant on the intervals  $\pi j/q< x< \pi(j+1)/q$ for $j=0, \ldots, 2q-1$, combined with a continuously differentiable function $-H(x)+C(t^\ast)x$, being composed of different parabolas defined on the intervals  $\pi j/q\leq x< \pi(j+1)/q, \ j=0, \ldots, 2q-1$.
Therefore, in the present case,  the solution profile is a discontinuous, piecewise parabolic curve.

For instance, let us take $t^\ast=\pi/3$ as an example.
On the one hand,
\begin{equation*}
\mathrm{I}\left(\frac{\pi}{3}, x\right)=-\frac{4}{\pi}\sum_{n=0}^{+\infty}\frac{\cos((2n+1)^2 \frac{\pi}{3})}{2n+1}\sin((2n+1) x)=\sum_{j=0}^{5}a_j\left(\frac{1}{3}\right)\sigma^{j, q}(x),\end{equation*}
where, a direct computation through IDFT shows that
\begin{equation*}
a_0=a_2=a_3=a_5=0,\quad a_1=-1,\quad a_4=1.\end{equation*}
On the other hand, $b_j\,( 0\leq j\leq 5)$ in $H(x)$ are obtained from the relation
\begin{equation*}
-\frac{4}{\pi}\sum_{n=0}^{+\infty}\frac{\sin((2n+1)^2 \frac{\pi}{3})}{2n+1}\sin((2n+1) x)=\sum_{j=0}^{5}b_j\left(\frac{1}{3}\right)\sigma^{j, q}(x).\end{equation*}
We have
\begin{equation*}
b_0=b_2=-b_3=-b_5=-\frac{\sqrt{3}}{3},\quad b_1=-b_4=-\frac{2\sqrt{3}}{3}.\end{equation*}
Again, owing to the periodicity of the solution, one has
\begin{equation}\label{R-series}
\sum_{n=0}^{+\infty}\frac{\sin((2n+1)^2 \frac{\pi}{3})}{(2n+1)^2}=-\frac{\pi}{4}C\left(\frac{\pi}{3}\right)=\frac{\sqrt{3}\,\pi^2}{18}.
\end{equation}
Finally, inserting them into \eqref{Hx} and then \eqref{sol-ra-beam} gives rise to the explicit solution at time $t^\ast=\pi/3$, namely
\begin{eqnarray}\begin{aligned}\label{sol-pi3-1}
u\left(\frac{\pi}{3}, x\right)
=
\begin{cases}
\frac{\sqrt{3}}{6} \left(x^2-\frac{4\pi}{3} x\right), \quad \quad &0\leq x\leq \frac{\pi}{3},\\
\frac{\sqrt{3}}{3} \left(x^2-\pi x+\frac{\pi^2}{18}\right)-1, \quad \quad &\frac{\pi}{3}\leq x\leq\frac{2\pi}{3},\\
\frac{\sqrt{3}}{6} \left(x^2-\frac{2\pi}{3} x-\frac{\pi^2}{3}\right), \quad \quad &\frac{2\pi}{3}\leq x\leq\pi,\\
-\frac{\sqrt{3}}{6} \left(x^2-\frac{10\pi}{3} x+\frac{7\pi^2}{3}\right), \quad \quad &\pi\leq x\leq\frac{4\pi}{3},\\
-\frac{\sqrt{3}}{3} \left(x^2-3\pi x+\frac{37\pi^2}{18}\right)+1, \quad \quad &\frac{4\pi}{3}\leq x\leq\frac{5\pi}{3},\\
-\frac{\sqrt{3}}{6} \left(x^2-\frac{8\pi}{3} x+\frac{4\pi^2}{3}\right), \quad \quad &\frac{5\pi}{3}\leq x\leq 2\pi.
\end{cases}
\end{aligned}\end{eqnarray}

\begin{remark}
Note that, at the rational points, the series \eqref{R-series} contains the odd terms of Riemann's non-differentiable function, which was introduced by Riemann in 1872; see \cite{Dui} for details. The connection between the Riemann's non-differentiable function and solutions of the vortex filament equation with polygonal initial data was recently established in \cite{HKV1, HKV2}.
\end{remark}

All in all, in the context of the linear beam equation,  the evolution of the periodic step function initial datum will take on three different qualitative behaviors. At irrational times, it evolves into continuous but non-differentiable fractal-like profile.  At rational times $t^\ast=\pi p/q\,(q\neq 2)$, the solution takes on a discontinuous, piecewise parabola behavior.  On the other hand, at each specific rational time $t^0_k=\pi(2k-1)/2, \,k\in\mathbb{Z^+}$, the quantization effect disappears entirely, and the solution instantly becomes a continuously differentiable function, emerging at regular $\pi$-periodic intervals. We conclude that
the revival phenomena exhibited by the periodic evolution of the linear beam equation differs from that arising in the linear KdV, the linear Schr\"{o}dinger, and other unidirectional linear dispersive evolution equations studied previously.

\begin{figure}[H]
    \centering
    \subfigure[$t=\pi/2$]{
    \includegraphics[width=0.3\textwidth]{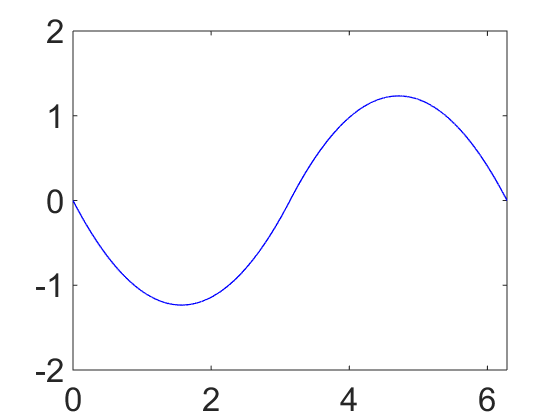}
}
    \subfigure[$t=\pi/3$]{
    \includegraphics[width=0.3\textwidth]{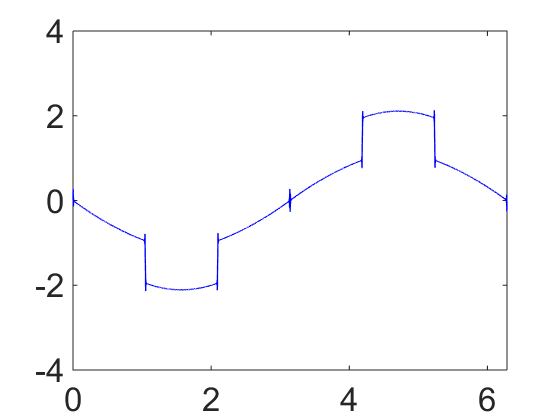}
}
    \subfigure[$t=\pi/5$]{
    \includegraphics[width=0.3\textwidth]{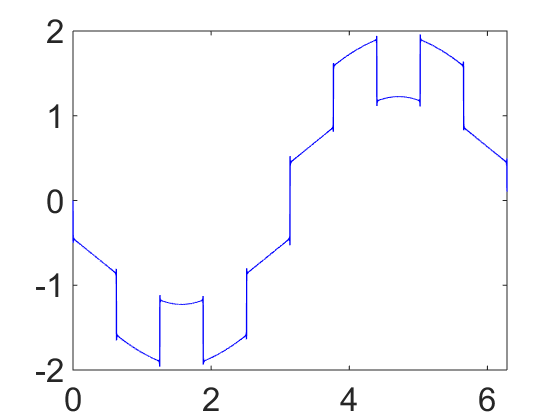}
}

     \caption{{\small The solutions  to the periodic initial-boundary value problem for the linear beam equation at rational times.}}
     \label{lin-beam-ra}
     \end{figure}

     \begin{figure}[H]
    \centering
    \subfigure[$t=0.1$]{
    \includegraphics[width=0.3\textwidth]{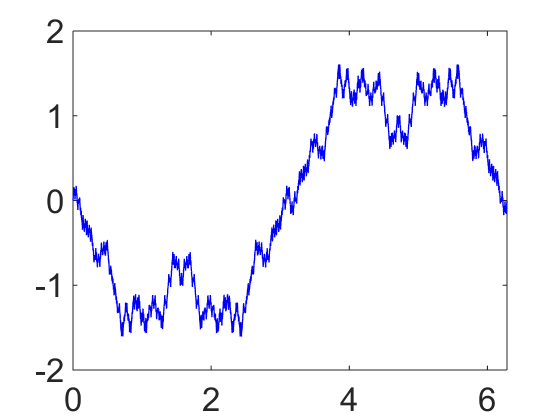}
}
    \subfigure[$t=0.3$]{
    \includegraphics[width=0.3\textwidth]{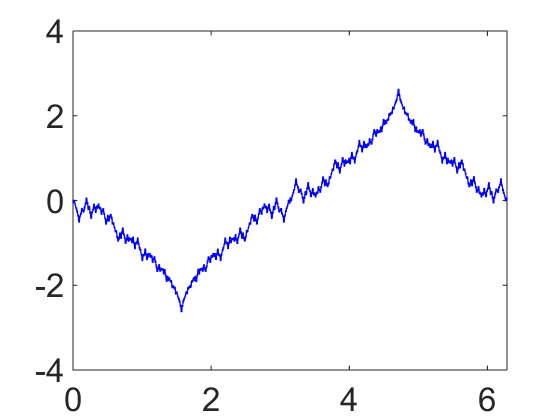}
}
    \subfigure[$t=0.5$]{
    \includegraphics[width=0.3\textwidth]{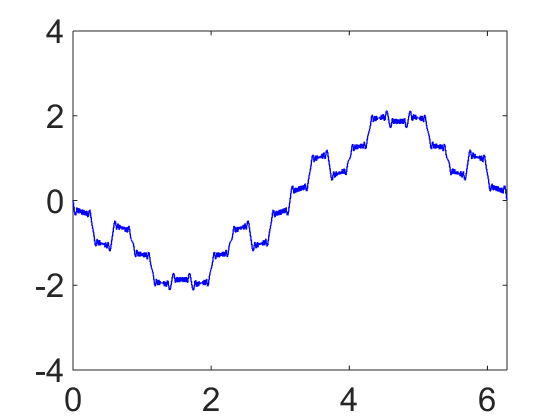}
}

     \caption{{\small The solutions  to the periodic initial-boundary value problem for the  linear beam equation at irrational times. }}
     \label{lin-beam-irra}
     \end{figure}

 In  Figures \ref{lin-beam-ra} and \ref{lin-beam-irra}, we display the graphs of the solution at some representative rational and   irrational times, respectively. These figures are plotted by straightforwardly applying the Fourier series representation of $u(t, x)$ given by \eqref{sol1-beam}. We sum over 1001 terms\footnote{Summing over a larger number of terms produces no appreciable difference in the solution profiles.} to obtain the numerical approximation of the solution. As illustrated in Figure \ref{lin-beam-ra}(a), the solution is  continuous at $\pi/2$. While, referring to Figure \ref{lin-beam-ra}(b) and  Figure \ref{lin-beam-ra}(c), it appears that, at $\pi/3$ and $\pi/5$, there exist a finite number of jump discontinuities, and between which parabolic curves of different form arise.  Obviously, the plots in Figure \ref{lin-beam-ra}, which obtained by simply truncating the Fourier series \eqref{sol1-beam}, are entirely consistent with the explicit expressions given by \eqref{sol-pi2} and \eqref{sol-pi3-1}. On the other hand, Figure \ref{lin-beam-irra} shows that, at irrational times, the solution displays continuous, but nowhere differentiable fractal-like profiles, as claimed above.

\begin{remark}\label{beam}
It is worth mentioning that, in view of the series \eqref{eq-s1-l1} and that in $\mathrm{II}^{(1)}(t^\ast, x)$, the distribution of the discontinuity points in $H(x)$ depends on the value of $q$, especially on its parity. As studied in \cite{OT18}, general speaking, the piecewise subintervals for these series are $\pi j/q\leq x< \pi(j+1)/q, j=0, \ldots, 2q-1$. However,   if $q$ is even (for instance $q=2$, whose corresponding solution is a representative example which can manifest such characteristic),  the solutions sometimes assume identical values on adjacent subintervals, and so exhibits larger regions of constancy. See also \cite{OT18} for a number-theoretic characterization of these occurrences.
\end{remark}

    \begin{figure}[H]
    \centering
    \subfigure[$t=\pi/2$]{
    \includegraphics[width=0.3\textwidth]{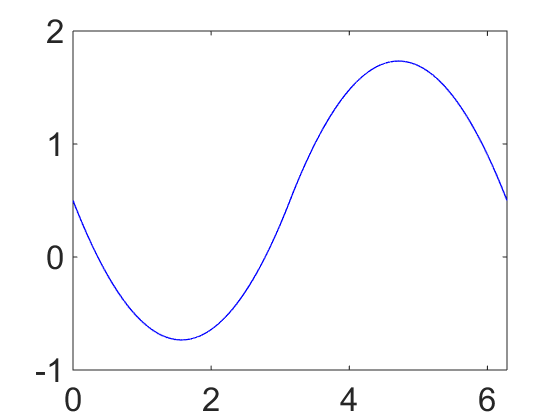}
}
    \subfigure[$t=\pi/3$]{
    \includegraphics[width=0.3\textwidth]{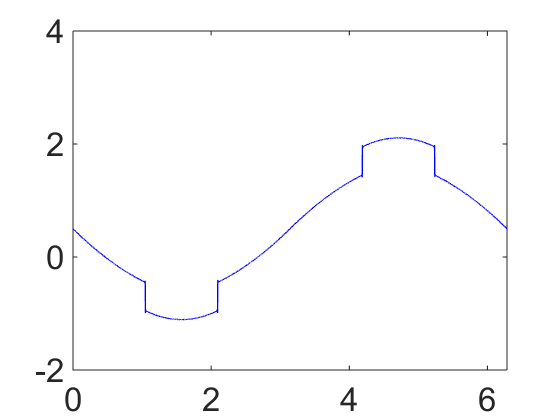}
}
    \subfigure[$t=\pi/5$]{
    \includegraphics[width=0.3\textwidth]{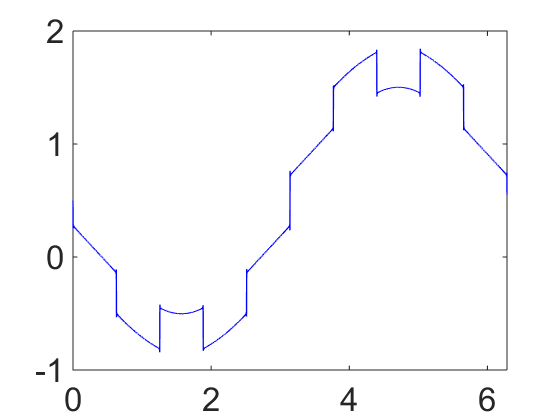}
}\\
 \subfigure[$t=0.1$]{
    \includegraphics[width=0.3\textwidth]{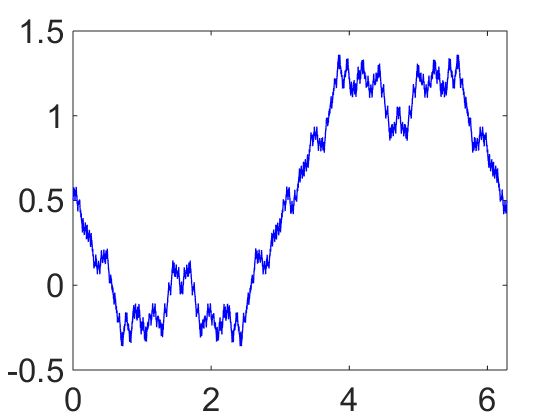}
}
    \subfigure[$t=0.3$]{
    \includegraphics[width=0.3\textwidth]{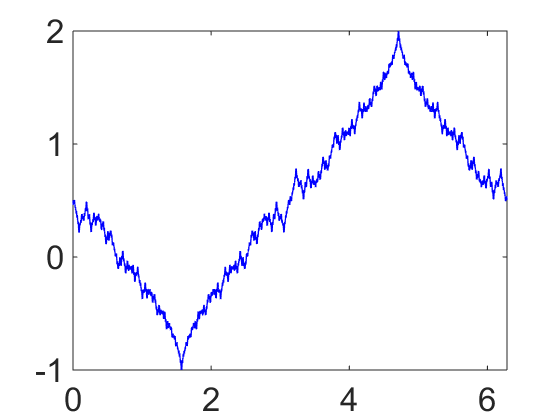}
}
    \subfigure[$t=0.5$]{
    \includegraphics[width=0.3\textwidth]{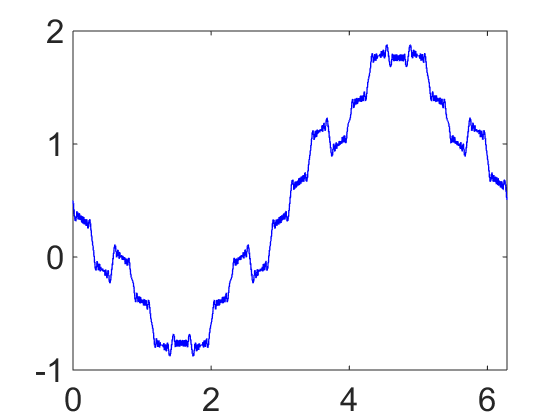}
}

     \caption{{\small The solutions  to the periodic initial-boundary value problem \eqref{ibv-beam} with initial data $f(x)=\tilde{\sigma}(x),\,g(x)=\sigma(x)$. }}
     \label{lin-beam-f01}
     \end{figure}

     \begin{figure}[H]
    \centering
    \subfigure[$t=\pi/2$]{
    \includegraphics[width=0.3\textwidth]{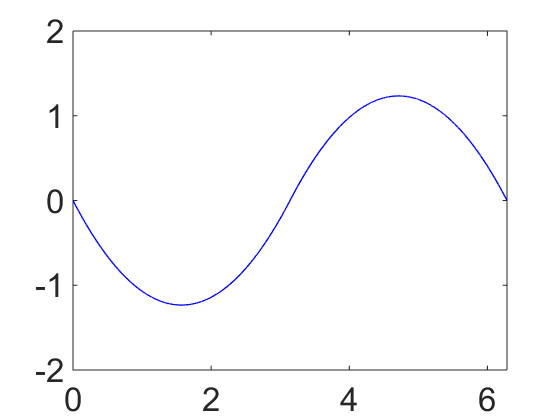}
}
    \subfigure[$t=\pi/3$]{
    \includegraphics[width=0.3\textwidth]{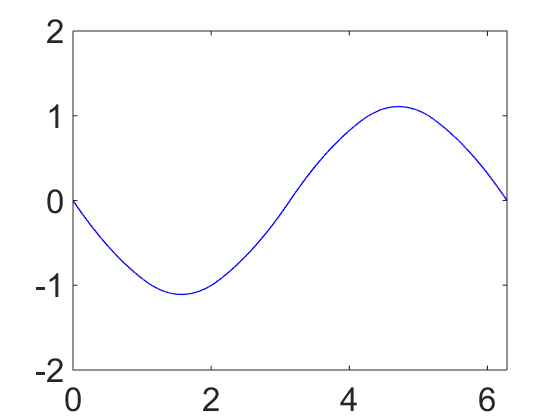}
}
    \subfigure[$t=\pi/5$]{
    \includegraphics[width=0.3\textwidth]{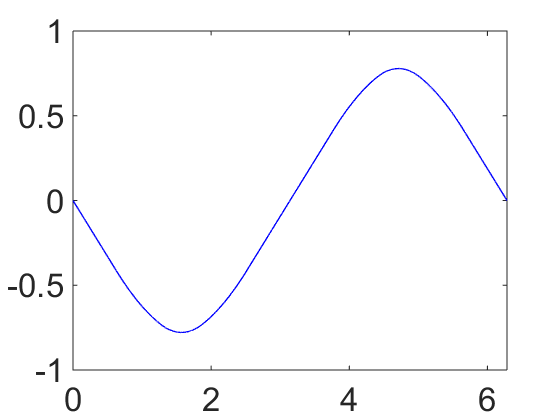}
}\\
 \subfigure[$t=0.1$]{
    \includegraphics[width=0.3\textwidth]{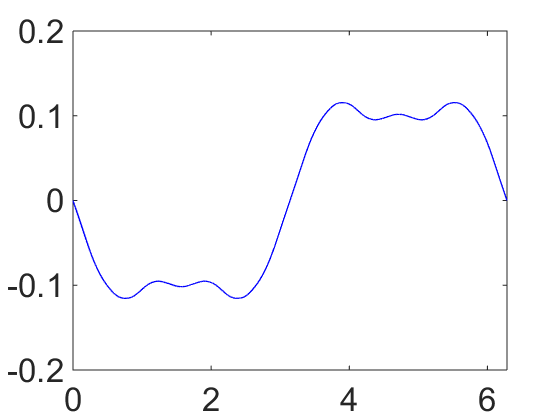}
}
    \subfigure[$t=0.3$]{
    \includegraphics[width=0.3\textwidth]{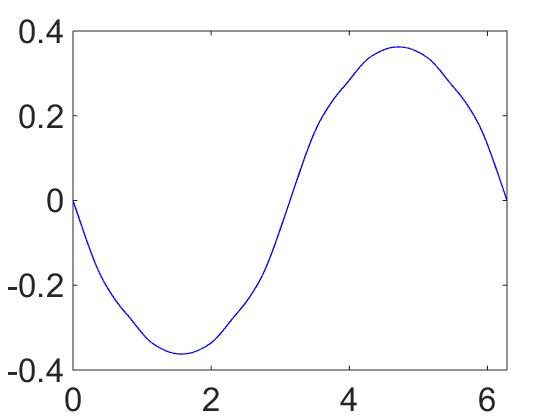}
}
    \subfigure[$t=0.5$]{
    \includegraphics[width=0.3\textwidth]{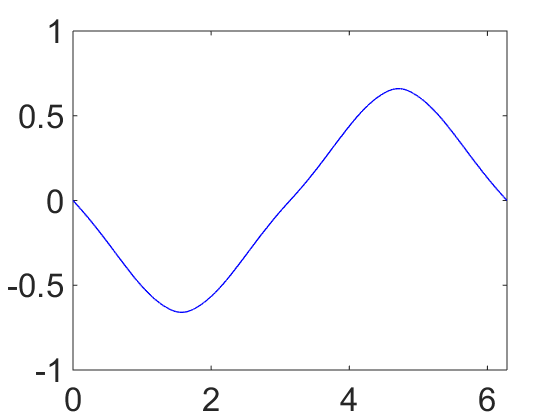}
}

     \caption{{\small The solutions  to the periodic initial-boundary value problem \eqref{ibv-beam} with initial data $f(x)=0,\,g(x)=\sigma(x)$. }}
     \label{lin-beam-f0}
     \end{figure}

 \begin{figure}[H]
    \centering
    \subfigure[$t=\pi/2$]{
    \includegraphics[width=0.3\textwidth]{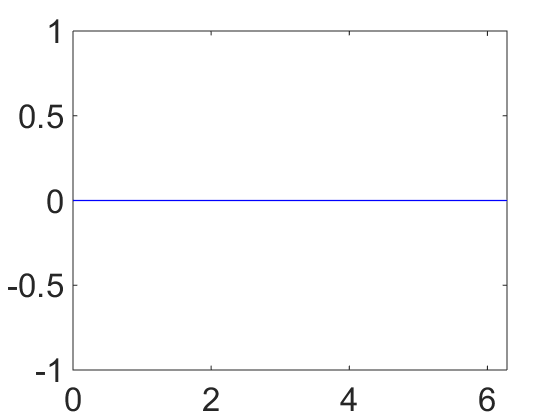}
}
    \subfigure[$t=\pi/3$]{
    \includegraphics[width=0.3\textwidth]{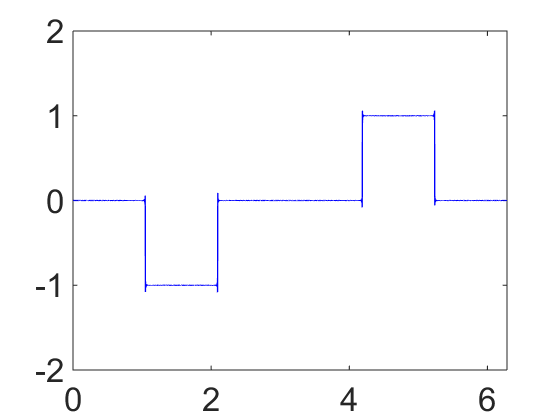}
}
    \subfigure[$t=\pi/5$]{
    \includegraphics[width=0.3\textwidth]{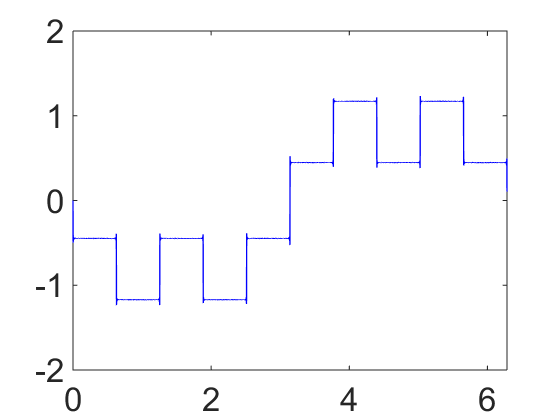}
}\\
 \subfigure[$t=0.1$]{
    \includegraphics[width=0.3\textwidth]{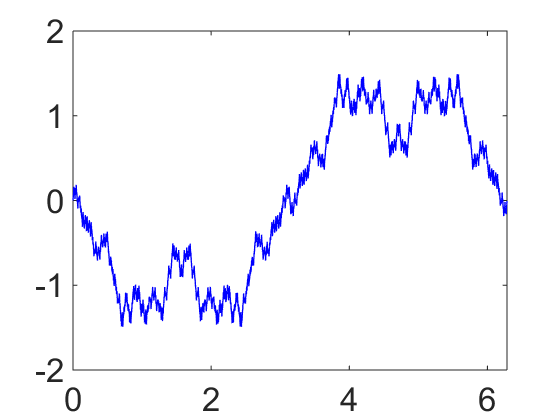}
}
    \subfigure[$t=0.3$]{
    \includegraphics[width=0.3\textwidth]{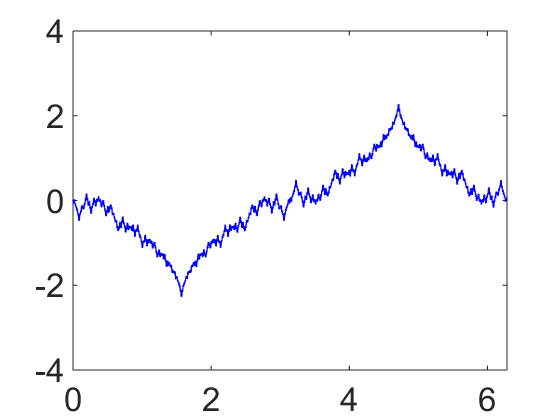}
}
    \subfigure[$t=0.5$]{
    \includegraphics[width=0.3\textwidth]{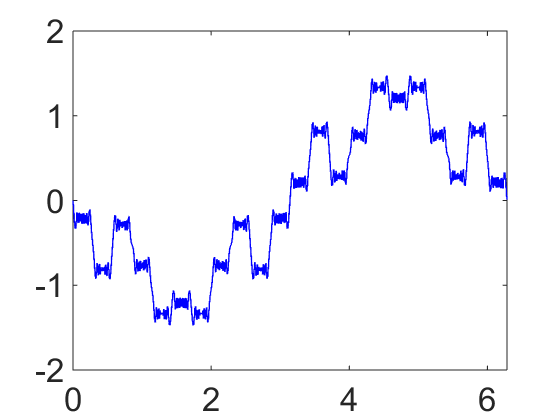}
}

     \caption{{\small The solutions  to the periodic initial-boundary value problem \eqref{ibv-beam} with initial data $f(x)=\sigma(x),\,g(x)=0$. }}
     \label{lin-beam-g0}
     \end{figure}

Formula \eqref{sol1-beam} provides the solution of \eqref{ibv-beam} with the same step function $f=g=\sigma(x)$ \eqref{iv-s} as the initial data. Indeed, in view of  Theorem \ref{thm-ge-eq} below, we find that the first and second terms in solution \eqref{sol1-beam} are induced by the initial data $u|_{t=0}=f(x)$  and $u_{t}|_{t=0}=g(x)$, respectively. It is noticed that the emergence of such a dichotomy phenomenon is not only for the case $f=g$. For instance, if we take the initial $g(x)=\sigma(x)$, while
 \begin{equation*}
 f(x)=\tilde{\sigma}(x)=
\begin{cases}
0 ,\qquad & 0\leq x<\pi,\\
1 ,\qquad & \pi\leq x<2\pi.
\end{cases}
\end{equation*}
 Figure \ref{lin-beam-f01} suggests that the different steps functions will also evolve into three different qualitative behaviors. Furthermore, in Figure \ref{lin-beam-f0} and Figure \ref{lin-beam-g0}, we display the graphs of solutions corresponding to $f(x)=0,\,g(x)=\sigma(x)$, and $f(x)=\sigma(x),\,g(x)=0$, respectively.  As demonstrated in Figure \ref{lin-beam-f0}, if $f(x)=0$, the dispersive quantization induced by $f(x)$ dispears entirely, and then the solution will retain a $C^1$ profile all the time. On the other hand, if  $f(x)=\sigma(x),\,g(x)=0$, referring to Figure \ref{lin-beam-g0}(b) and Figure \ref{lin-beam-g0}(c), the solutions take on dispersive quantization at the rational times. While, as shown in Figure \ref{lin-beam-g0}(a), the solution will vanish at $\pi/2$, since $\mathrm{I}(\pi/2, x)\equiv0$ as claimed above.

\section{Revival for bidirectional dispersive equations}

In the preceding section, we concentrated on the periodic initial-boundary problem for the linear beam equation. In this section,  we will seek to extend our analysis to bidirectional dispersive equations possessing more general dispersion relations,  as before, subject to general initial conditions and periodic boundary conditions.

Let $L$ be a scalar, constant coefficient integro-differential operator with real-valued Fourier transform
$\widehat{L}(k)=-\varphi(k)$, where $\varphi(k)>0$ is a real-valued function of $k$. The associated bidirectional scalar equation
\begin{equation}\label{ge-eq}
u_{tt}=L[u]\end{equation}
has real dispersion relation $\omega(k)=\pm \sqrt{\varphi(k)}$. We subject \eqref{ge-eq} to periodic boundary conditions posed on the interval $0\leq x \leq 2\pi$, and initial conditions
\begin{equation}\label{ge-ic}
u(0, x)=f(x),\quad  u_t(0, x)=g(x),
\end{equation}
where $f(x)$ and $g(x)$ are of bounded variation, and $g(x)$ is required to satisfy $\int_0^{2\pi} g(x)\;\mathrm{d}x=0$.

As usual, the first step is to construct the (formal) solution as a Fourier series
$$u(t,x) \sim \sum_{k=-\infty}^{+\infty}a_k(t)e^{\i kx}.$$
 To this end, we first expand the initial data $f(x)$ and $g(x)$ in Fourier series
\begin{equation*}\label{f-fs1}
f(x)\sim \sum_{k=-\infty}^{+\infty}c_ke^{\i k x},\quad \mathrm{where}\quad c_k=\widehat{f}(k)=\frac{1}{2\pi}\int_0^{2\pi} \, f(x)e^{-\i k x}\,\mathrm{d}x
\end{equation*}
and
\begin{equation*}\label{f-fs2}
g(x)\sim \sum_{k=-\infty}^{+\infty}d_ke^{\i k x},\quad \mathrm{where}\quad d_k=\widehat{g}(k)=\frac{1}{2\pi}\int_0^{2\pi} \, g(x)e^{-\i k x}\,\mathrm{d}x.
\end{equation*}
Next, an analogous analysis as used for the linear beam equation with step function initial conditions, implies that the corresponding coefficients $a_k(t)$ satisfy the following linear ODE
\begin{equation}\label{ak-ode-ge}
a_k''(t)+\varphi(k)a_k(t)=0.
\end{equation}
Solving it yields
\begin{equation*}
a_k(t)=A_ke^{\i \sqrt{\varphi(k)}t}+B_ke^{-\i \sqrt{\varphi(k)} t}.
\end{equation*}
Finally, using the initial data again, we find that the solution to the periodic initial-boundary value problem \eqref{ge-eq}-\eqref{ge-ic} is given by
\begin{eqnarray}\begin{aligned}\label{sol-ge-eq}
u(t, x)=\sum_{k}\widehat{f}(k)\cos\left(\sqrt{\varphi(k)}\> t\right)e^{\i kx}+\sum_{k\neq 0}\frac{\widehat{g}(k)}{\sqrt{\varphi(k)}}\sin\left(\sqrt{\varphi(k)}\>t\right)e^{\i kx}.
\end{aligned}\end{eqnarray}

With the Fourier series representation \eqref{sol-ge-eq} in hand, we are now able to analyze the qualitative behavior of the solution at rational times. We will show that the dynamical evolution of equation \eqref{ge-eq} on periodic domains with initial profiles \eqref{ge-ic} depends dramatically upon the asymptotics of the dispersion relation at large wave number. In all cases considered here, the large wave number asymptotics of the dispersion relation is given by a positive power of the wave number:
\begin{equation}
\sqrt{\varphi(k)}\sim \vert k\vert ^\alpha,\quad 2\leq \alpha\in\mathbb{R}, \quad\mathrm{as}\quad  \vert k\vert\rightarrow \infty.
\end{equation}

\subsection{Monomial dispersion relation:  }
As the first step, we will study the special case of monomial dispersion relation given by
\begin{equation}\label{monomial}
\omega(k)=\pm k^N,\quad 2\le N\in\mathbb{Z}^+.
\end{equation}
  The main results for the corresponding solutions are summarized in Theorem \ref{thm-ge-eq} below.
Hereafter, we define the operator $\partial_x^{-1}$  by the formula
\begin{equation}\label{minus1}
\partial_x^{-1}\,P(x)=\int_{2(k-1)\pi}^x\,P(y)\, \mathrm{d}y,\qquad x\in [ \,2(k-1)\pi, \,2k\pi \,].
\end{equation}
 We further define its $M$-th order power $\partial_x^{-M}$ via the recursive relation $\partial_x^{-M}=\partial_x^{-1}\partial_x^{-(M-1)}$, for $M\geq 1$.

\begin{theorem}\label{thm-ge-eq}
Suppose that equation \eqref{ge-eq} has the monomial dispersion relation \eqref{monomial}, the initial data $f(x)$ and $g(x)$  in \eqref{ge-ic} are of bounded variation, and $g(x)$ satisfies
$\int_0^{2\pi} g(x)\;\mathrm{d}x=0$. Let $G(x)=\partial_x^{-N}g(x)$. Then at each rational time $t^\ast=\pi p/q$, the solution to the periodic initial-boundary value problem  \eqref{ge-eq}-\eqref{ge-ic} takes the form
\begin{equation}\label{sol-ra-ge-eq}
u(t^\ast, x)=\sum_{j=0}^{2q-1}a_j\left(\frac{p}{q}\right)\,f\left(x-\frac{\pi j}{q}\right)+\i^N \sum_{j=0}^{2q-1}b_j\left(\frac{p}{q}\right)\,G\left(x-\frac{\pi j}{q}\right)+\sum_{j=0}^{N-1}C_jx^j,\end{equation}
\begin{equation}\label{Cj}
C_j=\frac{\i^j}{j\,!}\sum_{k\neq 0}\frac{\widehat{g}(k)\sin(k^Nt^\ast)}{k^{N-j}},
\end{equation}
where the coefficients $a_j, \,b_j\in \mathbb{C}$, $j=0,\ldots, 2q-1$, are  constants depending on $p$ and $q$.
\end{theorem}
The proof of the theorem relies on the following lemma, which is a direct corollary of Theorem 3.2 established in \cite{CO12} and is a  special case of Lemmas 7.5 and 7.6 in \cite{Fa}. Thus, we omit the proof. Moreover, we remark that the expression \eqref{sol-ra-ge-eq} for the exact solution is equivalent to \eqref{sol-ra-ge-eq 2} in the next subsection, although not exactly the same in form.

\begin{lemma}\label{Revival Representations}
Let $P(k)$ be an integral polynomial.
Assume that $f(x)$ is of bounded variation, and let $\widehat{f}(k)$ be the Fourier coefficient of $f(x)$, i.e.,
\begin{equation*}
\widehat{f}(k)=\frac{1}{2\pi}\int_0^{2\pi} \, f(x)e^{-\i k x}\,\mathrm{d}x.
\end{equation*}
Given $t^\ast=\pi p/q$, with $p$ and $0\neq q\in \mathbb{Z}^+$, there exist  constants $a_j^1, a_j^2\in \mathbb{C}$, $j=0,\ldots, 2q-1$, depending on $p$ and $q$, such the following two formulae hold:\\
\begin{equation}\label{eq1-l2}\sum_{k=-\infty}^{\infty}\widehat{f}(k)\cos\left(P(k) t^\ast\right) e^{\i kx}=\sum_{j=0}^{2q-1}a_j^1\left(\frac{p}{q}\right)f\left(x-\frac{\pi j}{q}\right),\end{equation}
\begin{equation}\label{eq2-l2}\sum_{k=-\infty}^{\infty}\widehat{f}(k)\sin\left(P(k) t^\ast\right) e^{\i kx}=\sum_{j=0}^{2q-1}a_j^2\left(\frac{p}{q}\right)f\left(x-\frac{\pi j}{q}\right).\end{equation}
\end{lemma}

\begin{proof} [\bf{Proof of Theorem \ref{thm-ge-eq}}]

First of all, according to \eqref{sol-ge-eq}, under the assumption of the theorem,  the solution to the corresponding periodic initial-boundary problem has the form
\begin{eqnarray}\begin{aligned}\label{sol-ge-eq-N}
u(t, x)&=\sum_{k}\widehat{f}(k)\cos(k^N t)e^{\i kx}+\sum_{k\neq 0}\frac{\widehat{g}(k)\sin(k^N t)}{k^N}e^{\i kx}:=\mathrm{I}(t, x)+\mathrm{II}(t, x).
\end{aligned}\end{eqnarray}
Furthermore, since $f(x)$ and $g(x)$ are of bounded variation and $N\geq 2$,  the first summation in expression \eqref{sol-ge-eq-N} is conditionally convergent, and the second one is absolutely convergent.

At the rational times $t^\ast=\pi p/q$, by Lemma 3.2, the first summation is a linear combination of translates of $f(x)$, i.e.,
\begin{equation*}
\mathrm{I}(t^\ast, x)=\sum_{j=0}^{2q-1}a^1_j\left(\frac{p}{q}\right)\,f\left(x-\frac{\pi j}{q}\right),
\end{equation*}
for certain $a_0^1, \ldots, a_{2q-1}^1\in \mathbb{C}$ determined by \eqref{eq1-l2} with $P(k)=k^N$.

Note that
\begin{equation}\label{parJ}
\partial_x^{-M}e^{\i kx}=\frac{1}{(\i k)^M}e^{\i kx}-\sum_{j=0}^{M-1}\frac{1}{j !\>(\i k)^{M-j}}x^j, \quad {\rm for} \quad 0 \leq x \leq 2 \pi.
\end{equation}
It follows that, at the rational times $t^\ast=\pi p/q$, the second summation satisfies
\begin{eqnarray*}\begin{aligned}
\mathrm{II}(t^\ast, x)&=\sum_{k}\i^N\,\widehat{g}(k)\sin(k^N t^\ast)\partial_x^{-N}e^{\i kx}+\sum_{k\neq 0} \widehat{g}(k)\sin(k^N t^\ast) \sum_{j=0}^{N-1}\frac{\i^j}{j! \>k^{N-j}}x^j\\
&:=\mathrm{II}^{(1)}(x)+\mathrm{II}^{(2)}(x).
\end{aligned}\end{eqnarray*}
Since $g(x)$ is of bounded variation, the series $C_j$ given in \eqref{Cj} is convergent for each $j=0, \ldots, N-1$, then the second component  $\mathrm{II}^{(2)}(x)$ readily leads to the last term in \eqref{sol-ra-ge-eq}.
On the other hand, in the case of $P(k)=k^N$, applying equation \eqref{eq2-l2} to the delta function $\delta(x)$  yields
\begin{equation*}\frac{1}{2\pi}\sum_{k=-\infty}^{\infty}\sin\left(k^N t^\ast\right) e^{\i kx}=\sum_{j=0}^{2q-1}b_j\left(\frac{p}{q}\right)\delta\left(x-\frac{\pi j}{q}\right),\end{equation*}
for some constants $b_j\in \mathbb{C},\; j=0, \ldots, 2q-1.$
We thus deduce $\mathrm{II}^{(1)}(x)$ as follows:
\begin{eqnarray*}\begin{aligned}
\mathrm{II}^{(1)}(x)&=\frac{\i^N}{2\pi}\sum_{k}\sin(k^N t^\ast)\int_0^{2\pi}g(y)e^{-\i ky}\partial_x^{-N}e^{\i kx}\,\mathrm{d}y\\
&=\frac{\i^N}{2\pi}\sum_{k}\sin(k^N t^\ast)\int_0^{2\pi}e^{\i ky}\partial_x^{-N}g(x-y)\,\mathrm{d}y\\
&=\i^N\int_0^{2\pi}G(x-y)\sum_{j=0}^{2q-1}b_j\delta\left(y-\frac{\pi j}{q}\right)\,\mathrm{d}y =\i^N\sum_{j=0}^{2q-1}b_jG\left(x-\frac{\pi j}{q}\right).
\end{aligned}\end{eqnarray*}
Summing  $\mathrm{II}^{(1)}(x)$, $\mathrm{II}^{(2)}(x)$ and $\mathrm{I}(t^\ast, x)$ gives \eqref{sol-ra-ge-eq}, which justifies the statement of the theorem.
\end{proof}

In particular, if the initial data $f(x)$ and $g(x)$ are the step function $\sigma(x)$ given in \eqref{iv-s},  the following corollary holds.

\begin{cor}\label{sol-ge-eqN-ivs}
Let $\sigma^{j, q}(x)$ be the box function defined in \eqref{box}. At a rational time $t^\ast=\pi p/q$, the solution to the periodic initial-boundary value problem \eqref{ge-eq}-\eqref{ge-ic} on the interval $0\leq x\leq 2\pi$, with initial data $f(x)=g(x)=\sigma(x)$ given in \eqref{iv-s} takes the form
\begin{equation}\label{sol-ra-geeq-sic}
u(t^\ast, x)=\sum_{j=0}^{2q-1}a_j\left(\frac{p}{q}\right)\sigma^{j, q}(x)+(-1)^{\left[\frac{N}{2}\right]}\partial_x^{-N}\sum_{j=0}^{2q-1}b_j\left(\frac{p}{q}\right)\sigma^{j, q}(x)+\sum_{j=0}^{\left[\frac{N}{2}\right]-1}D_j x^{2j+1},
\end{equation}
where
\begin{equation}\label{Dj}
D_j=\frac{(-1)^{j+1}4}{\pi(2j+1)!}\sum_{n=0}^{+\infty}\frac{\sin((2n+1)^Nt^\ast)}{(2n+1)^{N-2j}}, \qquad j=0, \ldots, \left[\frac{N}{2}\right]-1,
\end{equation}
and the coefficients $a_j, j=0,\ldots, 2q-1$ are determined by formula \eqref{eq1-l1} in Lemma 2.1,  $b_j, j=0,\ldots, 2q-1$ satisfy \eqref{eq2-l1} for even $N$, and \eqref{eq3-l1} for odd $N$, respectively.
\end{cor}

\begin{proof}
If $f(x)=g(x)=\sigma(x)$, the corresponding solution \eqref{sol-ge-eq} reduces to
  \begin{eqnarray*}\begin{aligned}\label{sol-beam}
u(t^\ast, x)&=-\frac{4}{\pi}\left[\>\sum_{n=0}^{+\infty}\frac{\cos((2n+1)^N t^\ast)\sin((2n+1) x)}{2n+1}+\sum_{n=0}^{+\infty}\frac{\sin((2n+1)^N t^\ast)\sin((2 n+1)x)}{(2n+1)^{N+1}}\>\right].
\end{aligned}\end{eqnarray*}
Obviously,  the first summation is exactly the first term in \eqref{sol-ra-geeq-sic}. As for the second summation,  a direct induction procedure shows that, if $N$ is even,
\begin{equation*}
\frac{\sin((2 n+1)x)}{(2n+1)^{N+1}}=(-1)^{\frac{N}{2}}\partial_x^{-N}\frac{\sin((2 n+1)x)}{2n+1}+\sum_{j=0}^{\frac{N}{2}-1}\frac{(-1)^{j}}{(2j+1)!(2n+1)^{N-2j}}x^{2j+1},
\end{equation*}
whereas, if $N$ is odd,
\begin{equation*}
\frac{\sin((2 n+1)x)}{(2n+1)^{N+1}}=(-1)^{\frac{N-1}{2}}\partial_x^{-N}\frac{\cos((2 n+1)x)}{2n+1}+\sum_{j=0}^{\frac{N-1}{2}-1}\frac{(-1)^{j}}{(2j+1)!(2n+1)^{N-2j}}x^{2j+1}.
\end{equation*}
Substituting into the second summation, and making use of formulae \eqref{eq2-l1} and \eqref{eq3-l1} for even and odd $N$, respectively, verifies \eqref{sol-ra-geeq-sic}, proving the corollary.
\end{proof}

More specifically, if the underlying equation is exactly the linear beam equation in \eqref{ibv-beam} with dispersion relation $\omega(k)=\pm k^2$, it follows that the second term in \eqref{sol-ra-geeq-sic} reduces to \eqref{II(1)-t1}, which is nothing but $-H(x)$ in \eqref{sol-ra-beam}. Meanwhile, the third term is identical to  $C(t^\ast)x$ in \eqref{sol-ra-beam}. This indicates that in this particular case, Corollary \ref{sol-ge-eqN-ivs} is in accordance with Theorem \ref{thm-beam}.

As in Section 2, let us now illustrate how, by Corollary~\ref{sol-ge-eqN-ivs}, we can calculate the value of the Riemann zeta function at $s=4$. We define
\begin{equation*}
H_{p, q}^N(x)=\partial_x^{-N}\sum_{j=0}^{2q-1}b_j\left(\frac{p}{q}\right)\sigma^{j, q}(x),
\end{equation*}
where $b_j$ are determined by the formulae \eqref{eq2-l1} for even $N$, or \eqref{eq3-l1} for odd $N$, respectively.
Denote
\begin{equation}\label{series}
S_{l}^N(t)=\sum_{n=0}^{+\infty}\frac{\sin((2n+1)^N t)}{(2n+1)^l},\quad\quad  \mathrm{for} \quad l\in\mathbb{Z^+}, \quad\mathrm{with}\quad  l\geq 2,
\end{equation}
and let
\begin{equation}\label{gamma}
\Gamma_N
=
\begin{cases}
\displaystyle\sum\limits_{k=1}^{\frac{N}{2}}\frac{(-1)^{k}(2\pi)^{N-2k+1}}{(N-2k+1)!}S_{2k}^N(t^\ast), \quad \quad & \mathrm{if} \;N\; \mathrm{even},\\
\displaystyle\sum\limits_{k=1}^{\frac{N-1}{2}}\frac{(-1)^{k}(2\pi)^{N-2k}}{(N-2k)!}S_{2k+1}^N(t^\ast), \quad \quad & \mathrm{if} \;N\geq 3\; \mathrm{odd}.
\end{cases}
\end{equation}
According to \eqref{sol-ra-geeq-sic}, we find a formula involving the sum  $\Gamma_N$, which along with the periodicity produces
\begin{equation}\label{SN}
\Gamma_N=\frac{\pi}{4}H_{p, q}^N(2\pi).
\end{equation}
Note that, if $N$ is even, at the special rational times $t^\ast_l=(2l-1)\pi/2,\,l\in \mathbb{Z^+}$,
\begin{equation*}S_N^N\left(t^\ast_l\right)=(-1)^{l-1}\sum_{n=0}^{+\infty}\frac{1}{(2n+1)^N},\end{equation*}
while, if $N$ is odd, $S_N^N\left(t^\ast_l\right)$ is a alternating series, namely,
\begin{equation*}S_N^N\left(t^\ast_l\right)=(-1)^{l-1}\sum_{n=0}^{+\infty}\frac{(-1)^n}{(2n+1)^N}.\end{equation*}
 Hereafter, we denote
 \begin{equation*}\sigma (N)=\sum_{n=0}^{+\infty}\frac{1}{(2n+1)^{N}},\quad \mathrm{for}\; \mathrm{even} \;N,\qquad \tau(N)=\sum_{n=0}^{+\infty}\frac{(-1)^n}{(2n+1)^{N}},\quad \mathrm{for}\; \mathrm{odd} \;N,\end{equation*}
 respectively.
Therefore, in the special rational times $t^\ast_l$ setting, \eqref{SN} establishes the recursion formulae for  $\sigma(2k)$ and  $\tau(2k+1)$ for each $k\in \mathbb{Z^+}$. More precisely,  for even $N$
\begin{equation*}\sigma(N)=\frac{(-1)^{\frac{N}{2}}}{8}\left(H_{1, 2}^N(2\pi)-\frac{4}{\pi}\sum_{k=0}^{\frac{N}{2}-1}\frac{(-1)^{k}(2\pi)^{N-2k+1}}{(N-2k+1)!}\sigma(2k)\right), \end{equation*}
or, for odd $N$,
\begin{equation*}\tau(N)=\frac{(-1)^{\frac{N-1}{2}}}{8}\left(H_{1, 2}^N(2\pi)-\frac{4}{\pi}\sum_{k=0}^{\frac{N-1}{2}-1}\frac{(-1)^{k}(2\pi)^{N-2k}}{(N-2k)!}\tau(2k+1)\right), \end{equation*}
 which are initiated by the series $\sigma(2)$ \eqref{C-pi2} for even $N$, or $\tau(3)$ for odd $N$, respectively.
As far as $\tau(3)$ is concerned, one can verify from \eqref{eq3-l1} for $N=3$ that
\begin{equation*}
\tilde{b}_0=-1,\quad \tilde{b}_1=\tilde{b}_2=1,\quad \tilde{b}_3=-1,\end{equation*}
which immediately yields
\begin{eqnarray*}\begin{aligned}\label{sol-pi3}
H_{1, 2}^3(x)
=
\begin{cases}
-\frac{1}{6} x^3, \quad \quad &0\leq x\leq \frac{\pi}{2},\\
\frac{1}{6} \left(x^3-3\pi x^2+\frac{3\pi^2}{2}x-\frac{\pi^3}{4}\right), \quad \quad &\frac{\pi}{2}\leq x\leq \frac{3\pi}{2},\\
-\frac{1}{6} \left(x^3-6\pi x^2+12\pi^2x-\frac{13\pi^3}{2}\right), \quad \quad &\frac{3\pi}{2}\leq x\leq 2\pi.
\end{cases}
\end{aligned}\end{eqnarray*}
We thus arrive at
\begin{equation*}
\tau(3)=\sum_{n=0}^{+\infty}\frac{(-1)^n}{(2n+1)^3}=\frac{\pi^3}{32}.
\end{equation*}
When it comes to $\zeta(4)$, we calculate from \eqref{b-2-pi2} that
\begin{eqnarray*}\begin{aligned}\label{sol-zeta4}
H_{1, 2}^4(x)
=
\begin{cases}
-\frac{1}{24} x^4, \quad \quad &0\leq x\leq \pi,\\
\frac{1}{24} \left(x^4-8\pi x^3+12\pi^2x^2-8\pi^3x+2\pi^4\right), \quad \quad &\pi \leq x\leq 2\pi.
\end{cases}
\end{aligned}\end{eqnarray*}
Consequently,
\begin{equation*}
\sigma(4)=\sum_{n=0}^{+\infty}\frac{1}{(2n+1)^4}=\frac{1}{8}\left(H_{1, 2}^4(2\pi)+\frac{4\pi^2}{3!}\bar{\zeta}(2)\right)=\frac{\pi^4}{96},\end{equation*}
which further yields the following classical result for the Riemann zeta function at $s=4$:
\begin{equation*}
\zeta(4)=\sum_{n=0}^{+\infty}\frac{1}{n^4}=\frac{\pi^4}{90}.\end{equation*}

\subsection{Monomial dispersion relation --- second approach}

We now briefly consider a different approach in the monomial case, which is based on \cite[Chapter 7]{Fa} and derive an alternative representation of the solution at rational times. Hence, the dispersion relation assumes the form \eqref{monomial}.
Moreover, only for this subsection, we relax the condition on $g$ and allow it to have non-zero mean over $[0,2\pi]$.

The solution to the periodic initial-boundary value problem \eqref{ge-eq}-\eqref{ge-ic} is given by
\begin{eqnarray}
	\begin{aligned}\label{sol-ge-eq-N1}
			u(t, x)&=\sum_{k}\widehat{f}(k)\cos(k^N t)e^{\i kx}+\frac{1}{2\pi}\int_{0}^{2\pi}g(y)\mathrm{d}y \ t + \sum_{k\neq 0}\frac{\widehat{g}(k)\sin(k^N t)}{k^N}e^{\i kx}\\
			&:=\mathrm{I}(t,x)+ \langle g \rangle \ t + \mathrm{II}(t,x),
	\end{aligned}
\end{eqnarray}
where $\langle g \rangle $ is the mean of $g$ and $\mathrm{I}(t,x)$, $\mathrm{II}(t,x)$ correspond to the two Fourier series representations respectively. In the following, we derive an alternative representation of the term $\mathrm{II}(t, x)$. In particular, we will show that $\mathrm{II}(t,x)$ can be expressed as the time-evolution of the periodic convolution of the function $g - \langle g \rangle$ with a polynomial of degree $N\geq 2$. As it is known, see for example \cite[Proposition 2.76]{Iorio}, the convolution gains the regularity of the most regular function between the two involved. Consequently, we may deduce that at any time $t>0$, either rational or irrational, the Fourier series representation of $\mathrm{II}(t,x)$ defines a $2\pi$-periodic function of class $C^{N}(\mathbb{R})$.

Let us make all the above precise and define first the following family of polynomials on $[0,2\pi]$.

\begin{definition}
	\label{QN-Pol-Def}
	Let $N\geq 1$ be an integer. We denote by $Q_{N}:[0,2\pi]\rightarrow \mathbb{C}$ the polynomial of degree $N$, defined inductively by the formula
	\begin{equation}
		\label{QN-Poly}
		Q_{N}(x) = \frac{(-\i)^{N} x^{N}}{(-1)^{N-1} N!} - \sum_{\ell=1}^{N-1} \frac{(-1)^{\ell - N}}{(-\i)^{\ell - N}} \frac{(2\pi)^{N-\ell}}{(N-\ell +1)!} Q_{\ell}(x).
	\end{equation}
\end{definition}

The crucial feature of the polynomial $Q_{N}$ is the form of its Fourier coefficients, which are equal to $k^{-N}$. As we shall shortly see, this  will allow us to invoke the operation of the periodic convolution in the representation of $\mathrm{II}(t,x)$.

\begin{lemma}
	\label{QN-Poly-FC}
	Fix an integer $N\geq 1$ and consider the polynomial $Q_{N}:[0,2\pi]\rightarrow \mathbb{C}$. Then, for $k\not=0$, $\widehat{Q_{N}}(k) = k^{-N}$.
\end{lemma}
\begin{proof}
	The proof follows by induction on $N$. It is easy to show that the statement holds for $N=1$ and $N=2$. We assume that $\widehat{Q_{\ell}}(k) = k^{-\ell}$ for $\ell=1,2,\dots, N$, with $N\geq3$, and calculate the Fourier coefficients of $Q_{N+1}$.
	
	Let $k\not=0$. Then, we have that
	\begin{equation*}
		\begin{aligned}
			\widehat{Q_{N+1}}(k) &= \frac{1}{2\pi}\int_{0}^{2\pi} Q_{N+1}(y) e^{-\i k y}dy \\
			&= \frac{(-i)^{N+1}}{2\pi (-1)^{N} (N+1)!}\int_{0}^{2\pi} y^{N+1}e^{-\i k y}dy - \sum_{\ell=1}^{N} \frac{(-1)^{\ell - N-1}}{(-\i)^{\ell - N-1}} \frac{(2\pi)^{N+1-\ell}}{(N-\ell)!} \frac{1}{k^{\ell}}.
		\end{aligned}
	\end{equation*}
	However, a direct calculation shows that
	\begin{equation}
		\int_{0}^{2\pi} y^{N+1}e^{-\i k y}dy = \frac{2 \pi (-1)^{N} (N+1)!}{(-\i)^{N+1} k^{N+1}} + (N+1)! \sum_{\ell = 1}^{N} \frac{(-1)^{\ell -1}}{(-\i)^{\ell}} \frac{(2\pi)^{N + 2-\ell}}{(N-\ell)!} \frac{1}{k^{\ell}}.
	\end{equation}
	Substituting back for $\widehat{Q_{N+1}}(k)$ we find that
	\begin{equation*}
		\begin{aligned}
			\widehat{Q_{N+1}}(k) = \frac{1}{k^{N+1}},
		\end{aligned}
	\end{equation*}
	which concludes the proof.
\end{proof}

We now turn our attention to the second ingredient needed for the alternative representation of $\mathrm{II}(t,x)$. Thus, we recall the definition of the periodic convolution, see \cite{Stein2011}.

\begin{definition}
	\label{Periodic Conv Def}
	Let $f$ and $g$ be $2\pi$-periodic on $\mathbb{R}$ and such that $f$, $g$ $\in L^{1}(0,2\pi)$. Then the \emph{$2\pi$-periodic convolution} of $f$ and $g$ is defined by
	\begin{equation}
		\label{Periodic Conv}
		f\ast g (x) = \frac{1}{2\pi}\int_{0}^{2\pi} f(x - y) g(y) dy, \quad x\in[0,2\pi].
	\end{equation}
\end{definition}

From \cite[Proposition 3.1]{Stein2011}, we know that $f\ast g$ defines a $2\pi$-periodic continuous function whose Fourier coefficients are given by $\widehat{f \ast g}(k) = \widehat{f}(k)  \widehat{g}(k)$. Summarizing the above, we arrive at the following lemma which identifies $\mathrm{II}(t,x)$ based on the convolution of $g - \langle g \rangle $ with $Q_{N}$.

\begin{lemma}
	\label{II representation lemma}
	Assume that $g$ is of bounded variation over $[0,2\pi]$. Fix integer $N\geq 2$ and consider the function
	\begin{equation}
		\label{v function}
		v(x) = (g-\langle g \rangle)\ast Q_{N} (x), \quad x\in [0,2\pi].
	\end{equation}
	Then, at any fixed time $t\geq 0$, we have that
	\begin{equation}
		\label{II representation}
		\mathrm{II}(t,x) = \sum_{k=-\infty}^{\infty} \widehat{v}(k)\sin(k^{N} t) e^{\i kx}.
	\end{equation}
\end{lemma}
\begin{proof}
	Let $\tilde{g} = g - \langle g \rangle$. Then, the Fourier coefficients of $\tilde{g}$ are given by
	\begin{equation*}
		\widehat{\tilde{g}}(k) = \widehat{g}(k), \quad k\not=0, \quad \widehat{\tilde{g}}(0) = 0.
	\end{equation*}
	Moreover, from Lemma~\ref{QN-Poly-FC}, we know that for $k\not=0$, $\widehat{Q_{N}}(k) = k^{-N}$. Hence,
	\begin{equation*}
		\widehat{v}(k) = \widehat{\tilde{g}}(k)\, \widehat{Q_{N}}(k) = \begin{cases}
			0, \quad k = 0, \\
			\frac{\widehat{g}(k)}{k^{N}}, \quad k\not=0,
		\end{cases}
	\end{equation*}
	which implies that for any $t>0$,
	\begin{equation*}
		\sum_{k=-\infty}^{\infty} \widehat{v}(k)\sin(k^{N} t) e^{\i kx} = \sum_{k\not=0} \frac{\widehat{g}(k)}{k^{N}}\sin(k^{N} t) e^{\i kx} = \mathrm{II}(t,x).
	\end{equation*}
\end{proof}

The validity of the revival effect at rational times $t^\ast=\pi p/q$ follows again by Lemma~\ref{Revival Representations} applied directly on $\mathrm{I}(t^\ast,x)$ and $\mathrm{II}(t^\ast,x)$ in conjunction with Lemma~\ref{II representation lemma}. This is the context of the next theorem.

\begin{theorem}\label{thm-ge-eq 2}
Suppose that equation \eqref{ge-eq} admits the monomial dispersion relation \eqref{monomial}, the initial data $f(x)$ and $g(x)$  in \eqref{ge-ic} are of bounded variation. Then, at each rational time $t^\ast=\pi p/q$, the solutions to the periodic initial-boundary value problem  \eqref{ge-eq}-\eqref{ge-ic} take the form
	\begin{equation}\label{sol-ra-ge-eq 2}
		u(t^\ast, x)=\sum_{j=0}^{2q-1}a_j\left(\frac{p}{q}\right)\,f\left(x-\frac{\pi j}{q}\right)+ \langle g \rangle \,t^{\ast} + \sum_{j=0}^{2q-1}d_j\left(\frac{p}{q}\right)\,v\left(x-\frac{\pi j}{q}\right),
	\end{equation}
	where $v(x) = ( g - \langle g \rangle ) \ast Q_{N} (x)$ and the coefficients $a_j, \,d_j\in \mathbb{C}$, $j=0,\ldots, 2q-1$, are certain constants depending on $p,q$.
\end{theorem}

As a consequence of Theorem~\ref{thm-ge-eq 2}, equivalently of Theorem~\ref{thm-ge-eq}, the solution at rational times is, at least, piecewise continuous, given that $f(x)$ has finitely many jump discontinuities. More specifically, the first term in \eqref{sol-ra-ge-eq 2}, $\mathrm{I}(t^\ast,x)$, corresponds to the revival of the initial function $f(x)$, whereas the third term, $\mathrm{II}(t^\ast)$, is the revival of the $2\pi$-periodic, $C^{N}(\mathbb{R})$ function $v(x) = ( g - \langle g \rangle ) \ast Q_{N} (x)$ and thus, together with the constant term, $\langle g \rangle t^{\ast}$, a $2\pi$-periodic, $C^{N}(\mathbb{R})$ function. Therefore, the solution is given as the sum of the revival of the initial condition $f(x)$ and a more regular function, which ensures the revival of the initial jump discontinuities of $f(x)$.

\subsection{Integral polynomial dispersion relation}
This subsection is concerned with the case that $\sqrt{\phi(k)}$ is an integral polynomial $P(k)$. The corresponding solution takes the form
\begin{eqnarray}\begin{aligned}\label{sol-ge-eq-Pk}
u(t, x)&=\sum_{k}\widehat{f}(k)\cos(P(k) t)e^{\i kx}+\sum_{k\neq 0}\frac{\widehat{g}(k)\sin(P(k) t)}{P(k)}e^{\i kx}:=\mathrm{I}(t, x)+\mathrm{II}(t, x).
\end{aligned}\end{eqnarray}
Firstly, using \eqref{eq1-l2} again, we obtain that  at each rational time $t^\ast=\pi p/q$, the first term in \eqref{sol-ge-eq-Pk} satisfies
\begin{equation*}
\mathrm{I}(t^\ast, x)=\sum_{j=0}^{2q-1}a_j\left(\frac{p}{q}\right)\,f\left(x-\frac{\pi j}{q}\right).
\end{equation*}
Next, notice that
\begin{eqnarray}\begin{aligned}
\left| \>\mathrm{II}(t, x)-\sum_{k\neq 0}\frac{\widehat{g}(k)\sin(P(k) t)}{c_Nk^N}e^{\i kx}\>\right|&\leq \sum_{k\neq 0}\frac{\vert c_{N-1}k^{N-1}+\cdots+c_{1}k+c_0 \vert }{\vert c_N P(k)k^N \vert}\vert \widehat{g}(k)\vert\\
&\lesssim \sum_{k\neq 0}\frac{1}{k^{N+2}},
\end{aligned}\end{eqnarray}
where the fact that $g(x)$ is of bounded variation has been used in the last inequality and $c_{N}$ denotes the coefficient of the highest power of $P(k)$. Since $N\geq 2$, the final series is absolutely convergent, whose sum is a constant. Thus, the above estimate implies that the qualitative behavior of the second term relies crucially on that of the series
\begin{equation}\label{se-II}
\sum_{k\neq 0}\frac{\widehat{g}(k)\sin(P(k) t)}{c_Nk^N}e^{\i kx}.
\end{equation}
While, as for \eqref{se-II}, a direct generalization of the proof of Theorem \ref{thm-ge-eq} shows that, at each rational time $t^\ast=\pi p/q$,
\begin{eqnarray*}\begin{aligned}
\sum_{k\neq 0}\frac{\widehat{g}(k)\sin(P(k) t^\ast)}{c_Nk^N}e^{\i kx}=\i^N\sum_{j=0}^{2q-1}\bar{b}_j\left(\frac{p}{q}\right)\,G\left(x-\frac{\pi j}{q}\right)+\sum_{j=0}^{N-1}\bar{C}_jx^j,
\end{aligned}\end{eqnarray*}
where the coefficients $\bar{b}_0, \ldots, \bar{b}_{2q-1}$ are determined by
\begin{equation*}\sum_{k}\sin\left(P(k) t^\ast\right) e^{\i kx}=\sum_{j=0}^{2q-1}\bar{b}_j\left(\frac{p}{q}\right)\delta\left(x-\frac{\pi j}{q}\right),\quad \hbox{and} \quad
\bar{C}_j=\sum_{k\neq 0}\frac{\widehat{g}(k)\sin(P(k)t^\ast)}{c_Nj! k^{N-j}}.
\end{equation*}
We thus conclude that, at each rational time, the series \eqref{se-II}  admits the same  discontinuities and revival structure as the second summation in solution \eqref{sol-ge-eq-N}.

All in all, we may safely draw the conclusion that, in the present case, the discontinuities of the solution will be determined by the initial data. For instance, if the initial data are the step function $\sigma(x)$, as in \eqref{iv-s}, by Corollary \ref{sol-ge-eqN-ivs}, $u(t, x)$ will be a $C^{N-1}$ curve at each $t_k^0=(2k-1)\pi/2,\;k\in \mathbb{Z}^+$, and exhibit jump discontinuities and revival profile at other rational times.

\subsection{Non-polynomial dispersion relation}
If $\sqrt{\varphi(k)}$ is not a polynomial, we distinguish two cases. The first one assumes that, for large wave numbers, the dispersion relation is asymptotically close to an integral polynomial $P(k)$. Hence, suppose
\begin{equation*}
\sqrt{\varphi(k)} \sim P(k)+O(k^{-1}),\quad \mathrm{as}\quad \vert k\vert\ \rightarrow \infty.
\end{equation*}
Firstly, under the assumption that $f(x)$ is of bounded variation, the first summation in \eqref{sol-ge-eq} satisfies
\begin{eqnarray}\begin{aligned}\label{non-poly-1}
&\left|\ \sum_{k}\widehat{f}(k)\cos\left(\sqrt{\varphi(k)}\> t\right)e^{\i kx}-\sum_{k}\widehat{f}(k)\cos(P(k) t)e^{\i kx}\ \right|\\
&\hskip1.5in\leq \sum_{k} \vert \widehat{f}(k)\vert\,\vert \cos\left(\sqrt{\varphi(k)}\> t\right)-\cos(P(k) t)\vert\lesssim \sum_{k\neq 0}\frac{1}{k^{2}},
\end{aligned}\end{eqnarray}
where the mean-value theorem has been used in the last inequality. Next, for the second summation, one has
\begin{eqnarray*}\begin{aligned}
&\left|\ \sum_{k\neq 0}\frac{\widehat{g}(k)}{\sqrt{\varphi(k)}}\sin\left(\sqrt{\varphi(k)}\> t\right)e^{\i kx}-\sum_{k\neq 0}\frac{\widehat{g}(k)}{k^N}\sin(P(k) t)e^{\i kx}\ \right|\\
&\hskip1in\leq \left|\ \sum_{k\neq 0}\frac{\widehat{g}(k)}{\sqrt{\varphi(k)}}\sin\left(\sqrt{\varphi(k)}\> t\right)e^{\i kx}-\sum_{k\neq 0}\frac{\widehat{g}(k)}{P(k)}\sin\left(\sqrt{\varphi(k)}\> t\right)e^{\i kx}\ \right|\\
&\hskip1.5in{}+\left|\ \sum_{k\neq 0}\frac{\widehat{g}(k)}{P(k)}\sin\left(\sqrt{\varphi(k)}\> t\right)e^{\i kx}-\sum_{k\neq 0}\frac{\widehat{g}(k)}{P(k)}\sin(P(k) t)e^{\i kx}\ \right|\\
&\hskip1.5in{}+\left|\ \sum_{k\neq 0}\frac{\widehat{g}(k)}{P(k)}\sin(P(k) t)e^{\i kx}-\sum_{k\neq 0}\frac{\widehat{g}(k)}{k^N}\sin(P(k) t)e^{\i kx}
\right|\\
&\hskip1in:=\mathrm{II}^{(1)}+\mathrm{II}^{(2)}+\mathrm{II}^{(3)}.
\end{aligned}\end{eqnarray*}
We directly estimate the above three terms as follows:
\begin{eqnarray*}\begin{aligned}
\mathrm{II}^{(1)}&\leq \sum_{k\neq 0}\frac{\vert O(k^{-1})\widehat{g}(k)\vert}{\vert P(k)\varphi(k)\vert}\lesssim \sum_{k\neq 0}\frac{1}{k^{2N+2}} ,
\\
\mathrm{II}^{(2)}&\leq \sum_{k\neq 0}\frac{\vert \sin(\sqrt{\varphi(k)} t)-\sin(P(k) t)\vert \vert \widehat{g}(k)\vert}{\vert P(k)\vert}\lesssim \sum_{k\neq 0}\frac{1}{k^{N+2}} ,
\\
\mathrm{II}^{(3)}&\leq \sum_{k\neq 0}\frac{\vert c_{N-1}k^{N-1}+\ldots+c_0\vert \vert \widehat{g}(k)\vert}{\vert P(k)k^N\vert}\lesssim \sum_{k\neq 0}\frac{1}{k^{N+2}} .
\end{aligned}\end{eqnarray*}
We thus conclude that
\begin{eqnarray*}\begin{aligned}
\left|\; \sum_{k\neq 0}\frac{\widehat{g}(k)}{\sqrt{\varphi(k)}}\sin(\sqrt{\varphi(k)} t)e^{\i kx}-\sum_{k\neq 0}\frac{\widehat{g}(k)}{k^N}\sin(P(k) t)e^{\i kx} \;\right|\lesssim \sum_{k\neq 0}\frac{1}{k^{N+2}},
\end{aligned}\end{eqnarray*}
which, together with the estimate \eqref{non-poly-1} imply that, in the present case, the solution $u(t, x)$ will exhibit the same asymptotic behavior as the polynomial case. The times at which the solution (approximately) exhibits revivals are densely embedded in the times at which it has a continuous, fractal profile.

For example, the Boussinesq equation
\begin{equation}\label{bs}
u_{tt}+\frac{1}{3}u_{xxxx}-u_{xx}+\frac{3}{2}\alpha(u^2)_{xx}=0,
\end{equation}
has the linear dispersion relation $\omega(k)=\pm k\sqrt{\frac{1}{3}k^2+1}$, and its leading order asymptotics is $\pm \frac{1}{\sqrt{3}}k^2$. The solution of the periodic initial-boundary value problem for the linearization of equation \eqref{bs} subject to the step function initial data \eqref{iv-s} at several representative rational and irrational times are plotted in Figure \ref{lin-bs}. As illustrated in these figures, the solutions  exhibit the (approximately) revival profile at rational times, and the overall jump discontinuities and revival structure is very similar to that of the linear beam equation.

\begin{figure}[H]
    \centering
    \subfigure[$t=\pi/2$]{
    \includegraphics[width=0.3\textwidth]{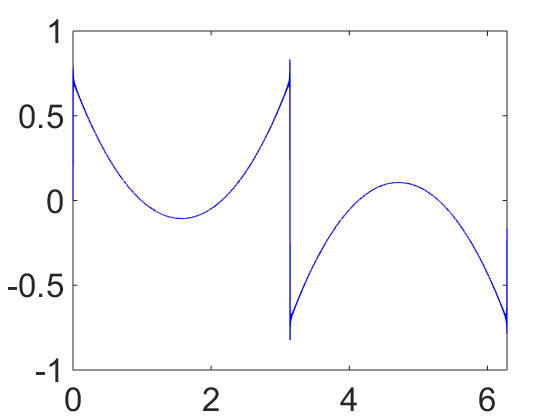}
}
    \subfigure[$t=\pi/3$]{
    \includegraphics[width=0.3\textwidth]{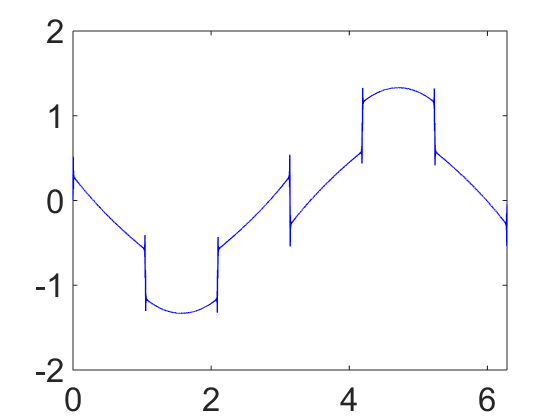}
}
    \subfigure[$t=\pi/5$]{
    \includegraphics[width=0.3\textwidth]{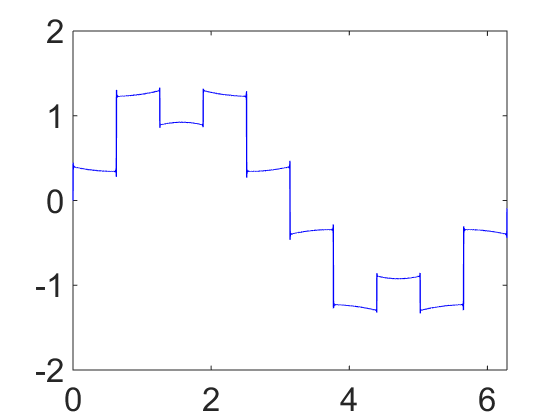}
}\\
 \subfigure[$t=0.1$]{
    \includegraphics[width=0.3\textwidth]{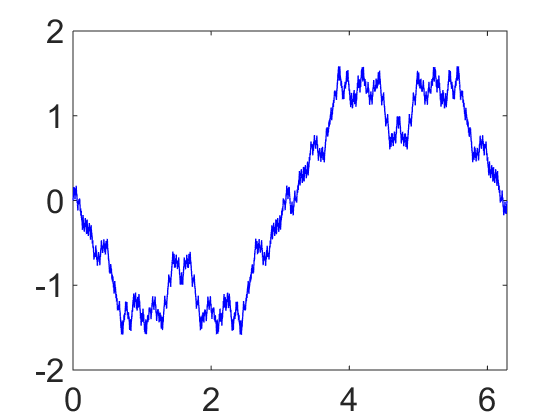}
}
    \subfigure[$t=0.3$]{
    \includegraphics[width=0.3\textwidth]{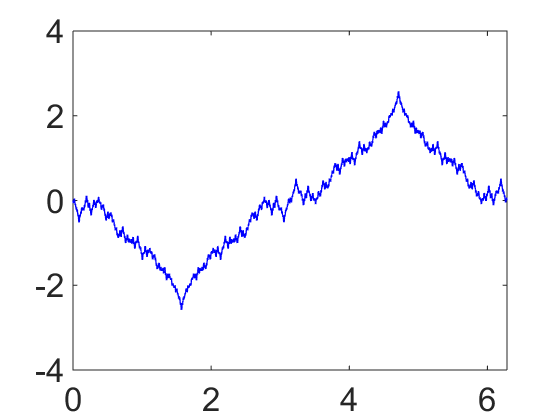}
}
    \subfigure[$t=0.5$]{
    \includegraphics[width=0.3\textwidth]{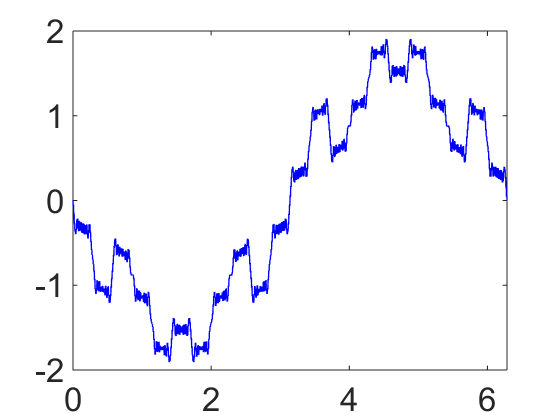}
}

     \caption{{\small The solutions  to the periodic initial-boundary value problem for the linear Boussinesq equation. }}
     \label{lin-bs}
     \end{figure}

On the other hand, if the equation admits a non-polynomial dispersion relation with an non-integral asymptotic exponent, i.e.,
\begin{equation*}
\sqrt{\varphi(k)}\sim \vert k\vert ^\alpha,\quad 2\leq \alpha\notin \mathrm{Z}, \;\mathrm{as}\quad  \vert k\vert\rightarrow \infty,
\end{equation*}
we still estimate the two summations in \eqref{sol-ge-eq} separately.
As far as the first one is concerned,  as studied in \cite{CO12}, its overall qualitative behavior is entirely determined by the asymptotic exponent $\alpha$. In particular, when $2\leq \alpha\notin \mathrm{Z}$ is not an integer, only fractal solution profiles will be observed at every time. While, as for the second term,
observe that in the present situation,
\begin{eqnarray*}\begin{aligned}
&\left|\ \sum_{k\neq 0}\frac{\widehat{g}(k)}{\sqrt{\varphi(k)}}\sin(\sqrt{\varphi(k)} t)e^{\i kx}-\sum_{k\neq 0}\frac{\widehat{g}(k)}{\vert k\vert ^{[\alpha]+1} }\sin(\vert k\vert ^\alpha t)e^{\i kx}\ \right|\\
&\hskip1.5in\leq \sum_{k\neq 0}\frac{\vert k\vert ^{[\alpha]+1}-\vert k\vert ^{\alpha}}{\vert k\vert ^{\alpha}\vert k\vert ^{[\alpha]+1}}\vert \widehat{g}(k)\vert \lesssim \sum_{k\neq 0}\frac{\vert k\vert ^{\alpha'}\ln \vert k\vert}{\vert k\vert ^{\alpha}\vert k\vert ^{[\alpha]+1}}\vert \widehat{g}(k)\vert,
\end{aligned}\end{eqnarray*}
for some $\alpha <\alpha'<[\alpha]+1$. Note that if $g(x)$ is of bounded variation, the estimate can only be obtained by using $\sum_{k\neq 0}{\vert k\vert ^{-[\alpha]-1}}$. In view of this situation, we need to further require that $g(x)$ satisfies
\begin{equation*}
\int\,e^{\i kx}\;\mathrm{d}g\sim O(k^{(\alpha-\alpha')-\delta }) \quad{\rm for\  all} \quad\delta > 0.
\end{equation*}
Under this hypothesis, the above estimate is bounded by $\sum_{k\neq 0}{\vert k\vert ^{-[\alpha]-2}}$, and hence the second term is completely determined by the series
\begin{equation*}
\sum_{k\neq 0}\frac{\widehat{g}(k)}{\vert k\vert ^{[\alpha]+1}}\sin(\vert k\vert ^\alpha t)e^{\i kx},
\end{equation*}
which, compared with the first term,  will admit better regularity.

\begin{figure}[H]
    \centering
    \subfigure[$t=\pi/10$]{
    \includegraphics[width=0.3\textwidth]{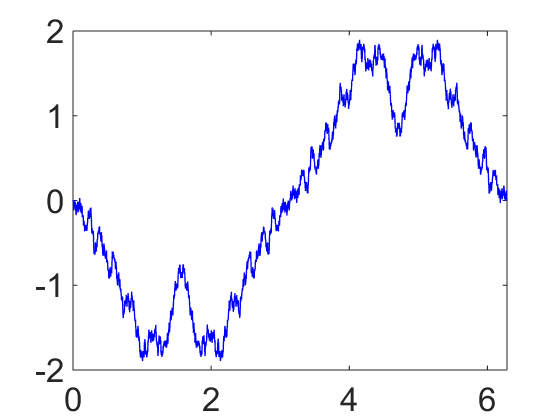}
}
    \subfigure[$t=\pi/5$]{
    \includegraphics[width=0.3\textwidth]{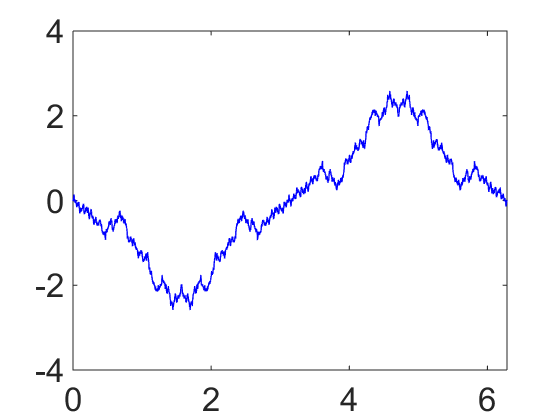}
}
    \subfigure[$t=\pi/3$]{
    \includegraphics[width=0.3\textwidth]{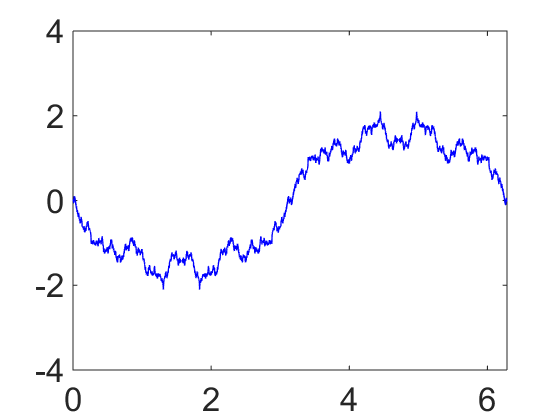}
}\\
 \subfigure[$t=0.1$]{
    \includegraphics[width=0.3\textwidth]{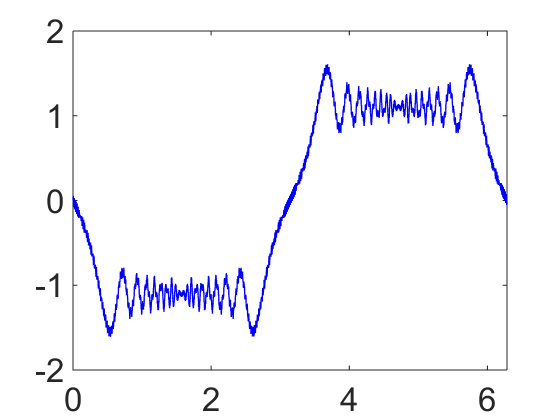}
}
    \subfigure[$t=0.3$]{
    \includegraphics[width=0.3\textwidth]{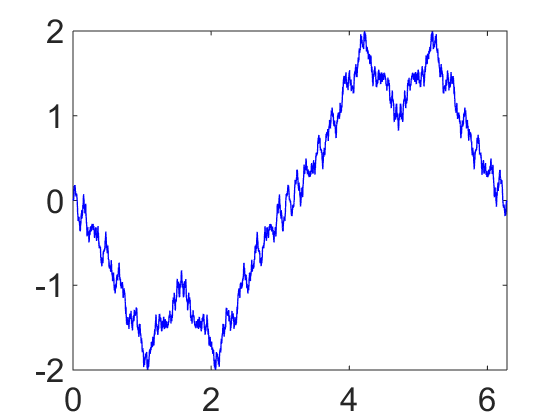}
}
    \subfigure[$t=0.5$]{
    \includegraphics[width=0.3\textwidth]{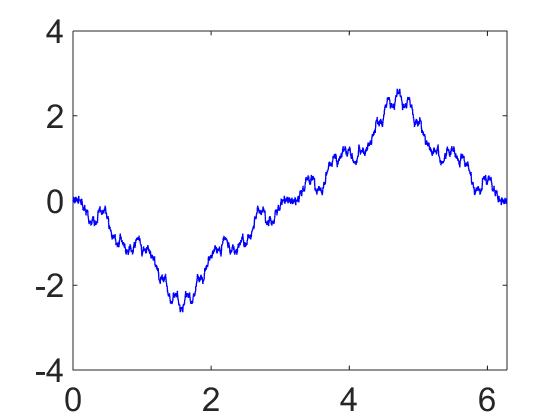}
}

     \caption{{\small The solutions  to the periodic Riemann problem for the linear equation \eqref{th-eq}.}}
     \label{bi-BO}
   \end{figure}

We conclude that, in the present case, the solutions will retain a fractal profile at all times. Results confirming this are displayed in Figure \ref{bi-BO}, which are the plots of the solutions to the periodic Riemann problem for the case of three-halves dispersion relation  $\omega(k)=\pm \vert k\vert^{\frac{3}{2}}$ corresponding to the equation
\begin{equation}\label{th-eq}
u_{tt}=\mathcal{H}[u_{xxx}],
\end{equation}
where $\mathcal{H}$ denotes the periodic Hilbert transform,
\begin{equation*}
\mathcal{H}[f](x)=\frac{1}{\pi}\sum_{-\infty}^{+\infty}\,\int_{-\pi}^{\pi}\hskip-19pt \hbox{---} \hskip10pt \frac{f(y)}{x-y+2\pi k}\,\mathrm{d}y =\frac1{2\pi}\, \int_{-\pi}^{\pi}{\hskip-19pt \hbox{---} \hskip10pt \cot\frac{x-y}2 f(y)}\,\mathrm{d}y .
\end{equation*}

\section{Numerical simulation of dispersive revival for  nonlinear equations}

In this section, we will explore the effect of periodicity on rough initial data for nonlinear equations in the context of the nonlinear defocusing cubic beam equation of the form
\begin{equation}\label{non-beam}
u_{tt}+u_{xxxx}+\mu \,u+ \varepsilon \,\vert u\vert^2u=0,
\end{equation}
which is motivated by the nonlinear Boussinesq equation, see \cite{MMS} for details. We will numerically approximate the solutions to the periodic initial-boundary value problem for the  beam equation \eqref{non-beam}  subject to periodic boundary conditions on  $[-\pi,\,\pi]$, with the same step function \eqref{iv-s} as initial data.

The goal of this section is to investigate to what extent revival and fractalization phenomena persist into the nonlinear regime.
A basic numerical technique, the Fourier spectral method, will be employed to approximate the solution to this initial-boundary value problem. As we will see, our numerical studies strongly indicate that the dispersive revival phenomenon admitted by the associated linearized equation will persist into the nonlinear regime. However, some of the qualitative details --- for instance, the convexity of the curves between the jump discontinuities --- will be affected by the nonlinearity, in contrast to what was observed in the unidirectional case.

\subsection{ The Fourier spectral method}
Let us first summarize the basic ideas behind the Fourier spectral method for approximating the solutions to nonlinear equations. One can refer to \cite{GO, Tre} for details of the method.

Formally, consider the initial value problem for a nonlinear evolution equation
\begin{equation}\label{ns-eq}
u_t=K[u],\qquad  u(0, x)=u_0(x),
\end{equation}
 where $K$ is a differential operator in the spatial variable with no explicit time dependence. Suppose $K$  can be written as $K=L+N$, in which $L$ is a linear operator characterized by its Fourier transform $\widehat{Lu}(k)=\omega(k)\widehat{u}(k)$, while $N$ is a nonlinear operator. We use $\mathcal{F}[\cdot]$ and $\mathcal{F}^{-1}[\cdot]$ to denote the Fourier transform and the inverse Fourier transform of the indicated function, respectively, so that the Fourier transform for equation in \eqref{ns-eq} takes the form
$$\widehat{u}_t=\omega(k)\widehat{u}+\mathcal{F}[\,N(\mathcal{F}^{-1}[\widehat{u}])\,].$$
Firstly, periodicity and discretization of the spatial variable enables us to apply the fast Fourier transform (FFT) based on, for instance, 512 space nodes, and arrive at a system of ordinary differential equations (ODEs), which we solve numerically. For simplicity, we adopt a uniform time step $0<\Delta t\ll1$, and seek to approximate the solution $\hat{u}(t_n)$ at the successive times $t_n=n\Delta t$ for $n=0, 1, \ldots$. The classic fourth-order Runge-Kutta method, which has a local truncation error of $O((\Delta t)^5)$, is adopted, and its iterative scheme is given by
\begin{equation*}
\widehat{u}(t_{n+1})=\widehat{u}(t_n)+\frac{1}{6}(f_{k_1}+2f_{k_2}+2f_{k_3}+f_{k_4}),\quad n=0, 1, \ldots, \quad \widehat{u}(t_0)=\widehat{u}_0(k),
\end{equation*}
where
\begin{eqnarray*}\begin{aligned}
f_{k_1}&=f(t_n, \widehat{u}(t_n)),&
f_{k_2}&= f(t_n+\Delta t/2,\widehat{u}(t_n)+\Delta t f_{k_1}/2),\\
f_{k_3}&= f(t_n+\Delta t/2, \widehat{u}(t_n)+\Delta t f_{k_2}/2),&\qquad
f_{k_4}&= f(t_n+\Delta t, \widehat{u}(t_n)+\Delta t f_{k_3})
\end{aligned}\end{eqnarray*}
with
\begin{eqnarray*}
f(t, \widehat{u})=\omega(k)\widehat{u}+\mathcal{F}[\,N(\mathcal{F}^{-1}[\widehat{u}])\,].
\end{eqnarray*}
Accordingly, the approximate solution $u(t, x)$ can be obtained through the inverse discrete Fourier transform.

Since the Runge-Kutta method is designed for first order systems of ordinary differential equations, we convert our bidirectional second order in time system  \eqref{non-beam} into a first order system by setting
\begin{equation*}
v=u_t,
\end{equation*}
an hence the beam equation \eqref{non-beam} is mapped to the following evolutionary system
 \begin{equation}\label{non-beam-sys}
u_{t}=v,\qquad v_t=-u_{xxxx}-\mu \,u-\varepsilon \,\vert u\vert^2u.
\end{equation}
The Fourier transform for \eqref{non-beam-sys} takes the form
\begin{equation}\label{non-beam-sys-ft}
\widehat{u}_{t}=\widehat{v},\qquad \widehat{v}_t=-( k)^4\widehat{u}-\mu \,\widehat{u}-\varepsilon \,\mathcal{F}[\,\vert \mathcal{F}^{-1}[\widehat{u}] \vert^2\mathcal{F}^{-1}[\widehat{u}]\,].
\end{equation}
Using the classic fourth-order Runge-Kutta method to solve the resulting system \eqref{non-beam-sys-ft}, and then taking the inverse discrete Fourier transform,  one can obtain the numerical solution to the periodic initial-boundary value problem for the nonlinear beam equation \eqref{non-beam}.

\subsection{ Numerical Results} Figure \ref{non-beam-ra} and Figure \ref{non-beam-irra} display some results from our numerical approximations of the solutions to the nonlinear beam equation \eqref{non-beam} with periodic boundary conditions and initial conditions \eqref{iv-s} at some representative rational and irrational times.
Comparing the graphs in these two figures with the graphs corresponding to the same times in Figure 1 and Figure 2, we find that, at each irrational time, all sets of plots are fairly similar to those from the associated linear beam equation,  and the solution still takes a continuous, non-differentiable profile. When it comes to the rational times, the same jump discontinuities consistency for nonlinear and linear equations emerges as well. Meanwhile, closer inspection will reveal some differences. The most noticeable is that, the shape of the curves between jump discontinuities will change with time evolution. More precisely, the graphs corresponding to $t=\pi/5$ show that  the differences of the solution profile between the linear and nonlinear equations is slight, except that,  in the nonlinear case,  the curves between the jumps become closer to constants.
Further, as the power $p$ decreases with increasing time, the variation in the shape of the curves  from linear to nonlinear becomes greater and greater. As illustrated in the graphs corresponding to $t=\pi/3$ and $t=\pi/2$, the convexity of the curves has completely changed. These differences of the qualitative behavior of the solutions exhibit the effect of the nonlinearity.
Furthermore, in order to better understand the effect of the nonlinearity, we perform further numerical experiments for smaller values of coefficients $\varepsilon$ and $\mu$ in equation \eqref{non-beam}. Referring to Figure \ref{non-beam-pi3}, it appears that solution at $t=\pi/3$ tends to the linear profile as $\varepsilon$ tends to zero. Meanwhile, the shape of the curves between jump discontinuities will change as $\varepsilon$ increases, the most noticeable variation being the changes in convexity. More unexpected phenomena appear when $t=\pi/2$.  We find the variation of the profile of the solution will be affected not only by the nonlinear term but also by the linear  term involving $u$. The plots displayed in Figure \ref{non-beam-pi2}, corresponding to some representative coefficients  $\varepsilon$ and $\mu$, suggest that the solution profile, including its convexity and the values of its peak and trough will be affected by the combination of both coefficients $\varepsilon$ and $\mu$.

\begin{figure}[H]
    \centering
    \subfigure[$t=\pi/2$]{
    \includegraphics[width=0.3\textwidth]{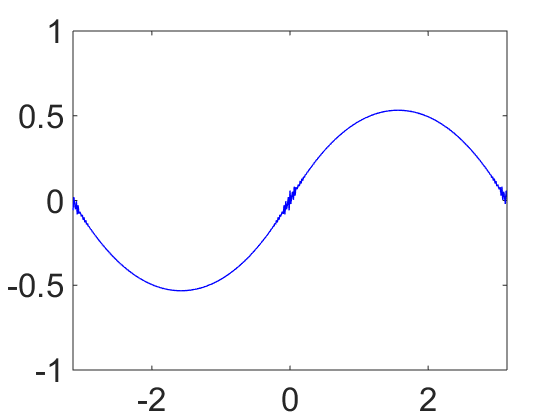}
}
    \subfigure[$t=\pi/3$]{
    \includegraphics[width=0.3\textwidth]{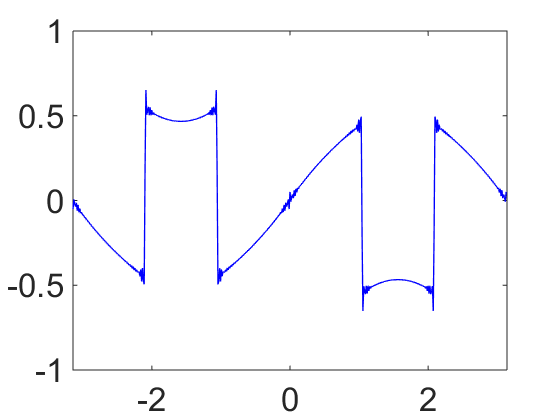}
}
    \subfigure[$t=\pi/5$]{
    \includegraphics[width=0.3\textwidth]{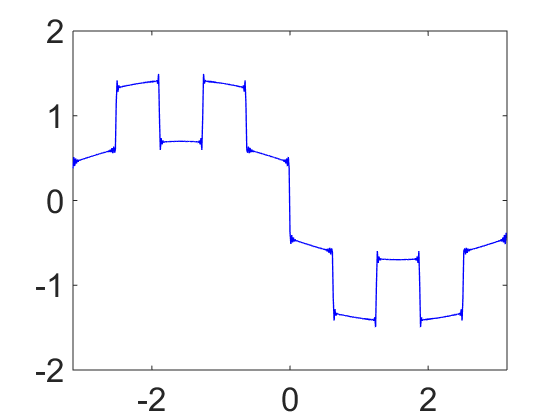}
}
 \caption{{\small The solutions  to the periodic initial-boundary value problem for the beam equation with $\mu=\varepsilon = 1$ at rational times. }}
    \label{non-beam-ra}
     \end{figure}

\vglue-.3in

\begin{figure}[H]
    \centering
 \subfigure[$t=0.1$]{
    \includegraphics[width=0.3\textwidth]{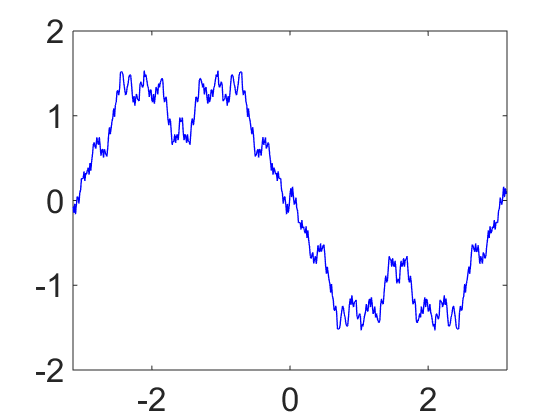}
}
    \subfigure[$t=0.3$]{
    \includegraphics[width=0.3\textwidth]{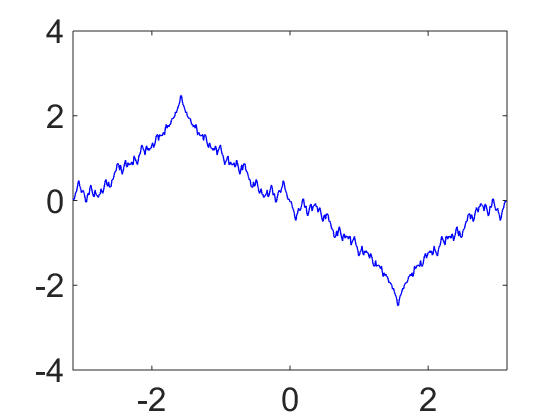}
}
    \subfigure[$t=0.5$]{
    \includegraphics[width=0.3\textwidth]{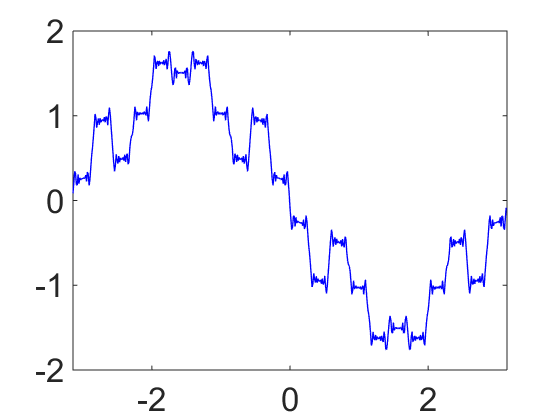}
}

\caption{{\small The solutions  to the periodic initial-boundary value problem for the beam equation with $\mu=\varepsilon = 1$ at irrational times.}}
  \label{non-beam-irra}
        \end{figure}

        \begin{figure}[H]
    \centering
    \subfigure[$\mu=1,\,\varepsilon=0.01$]{
    \includegraphics[width=0.3\textwidth]{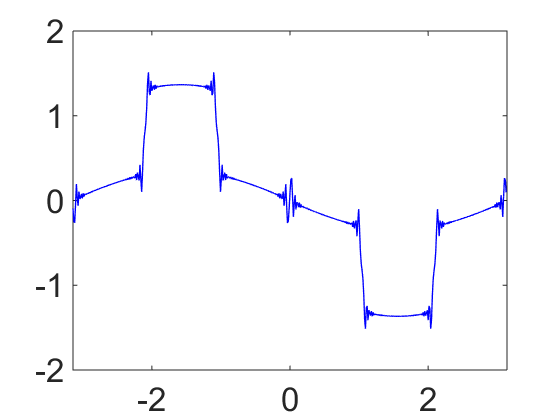}
}
    \subfigure[$\mu=1,\,\varepsilon=0.5$]{
    \includegraphics[width=0.3\textwidth]{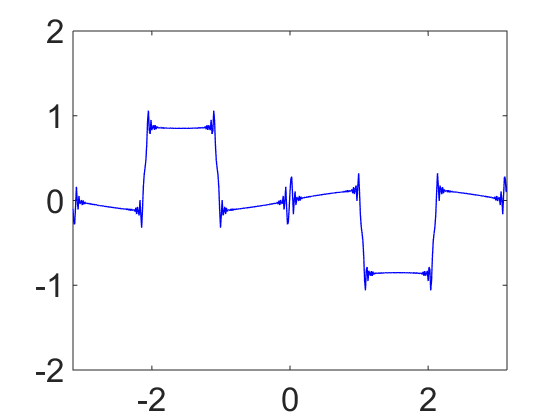}
}
    \subfigure[$\mu=1,\,\varepsilon=0.9$]{
    \includegraphics[width=0.3\textwidth]{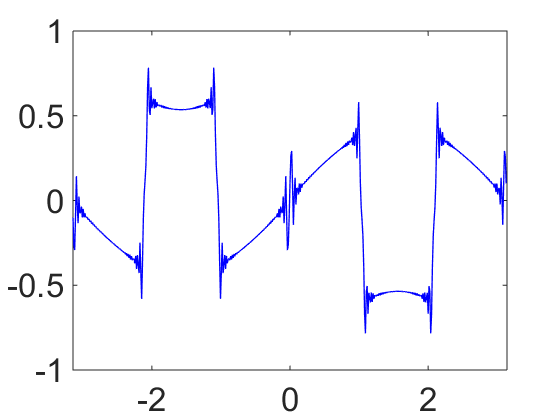}
}
 \caption{{\small The solutions  to the periodic initial-boundary value problem for the beam equation at $t=\pi/3$. }}
    \label{non-beam-pi3}
     \end{figure}

\vglue-.3in

\begin{figure}[H]
    \centering
 \subfigure[$\mu=0.001,\,\varepsilon=0.001$]{
    \includegraphics[width=0.3\textwidth]{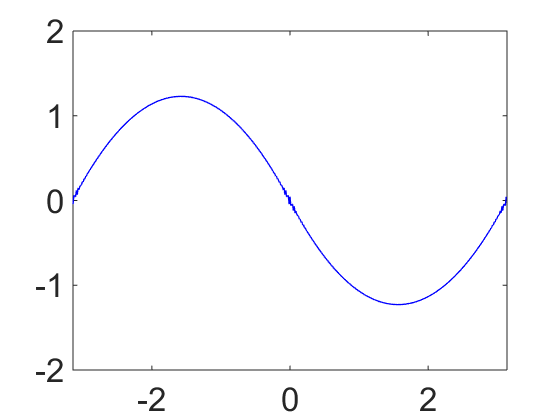}
}
    \subfigure[$\mu=0.001,\,\varepsilon=1$]{
    \includegraphics[width=0.3\textwidth]{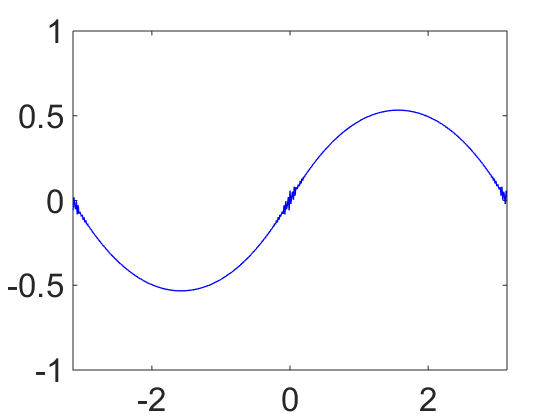}
}
    \subfigure[$\mu=1,\,\varepsilon=0.001$]{
    \includegraphics[width=0.3\textwidth]{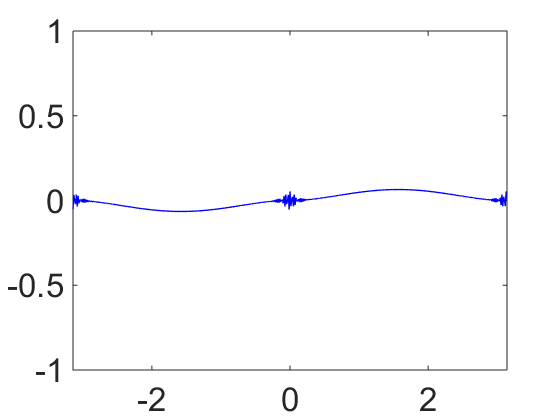}
}

\caption{{\small The solutions  to the periodic initial-boundary value problem for the beam equation at $t=\pi/2$.}}
  \label{non-beam-pi2}
        \end{figure}

Recall that the numerical experiments to the periodic initial-boundary value problem for  the KdV equation, the NLS equation and the multi-component KdV system have been previously analyzed in \cite{CO14, YKQ}, which show that,  in the unidirectional regime, the effect of the nonlinear flow can be regarded as a perturbation of the linearized flow. When it comes to the bidirectional dispersive equations,  our numerical simulation strongly indicates that,  the dichotomy of revival/fractalization at rational/irrational times  in linearization will persist into the nonlinear regime, and the finite ``revival'' nature of the solutions at rational times is not affected by the nonlinearity, however, the influence of the nonlinearity on the qualitative behavior of the solutions is much greater than in the unidirectional setting. Motivated by this observation, formulation of theorems and rigorous proofs concerning this novel revival phenomenon in the nonlinear bidirectional regime, specially for the nonlinear beam and Boussinesq equations,  is eminently worth further study.

\vskip 0.5cm
\noindent {\bf Acknowledgments.} Part of Farmakis' research was conducted during his Ph.D studies which were supported by the Maxwell Institute Graduate School in Analysis and its Applications, a Centre for Doctoral Training funded by EPSRC (grant EP/L016508/01), the Scottish Funding Council, Heriot-Watt University and the University of Edinburgh. Kang's research was supported by the  NSF of China under Grant-12371252 and Basic Science Program of Shaanxi Province under Grant-2019JC-28. Qu's research was supported by the NSF of China under Grant-12431008. Yin's research was supported by Northwest University Youbo Funds 2024006.   

$\,$

\vglue 0.5cm

\end{document}